\documentclass[11pt]{article}

\usepackage{amsmath}
\usepackage{amssymb}
\usepackage{amsthm}

\usepackage{graphicx} 

\usepackage[hidelinks]{hyperref}

\usepackage{subcaption}

\numberwithin{equation}{section}

\theoremstyle{theorem}
\newtheorem{theorem}{Theorem}[section]
\newtheorem{corollary}[theorem]{Corollary}
\newtheorem{lemma}[theorem]{Lemma}

\theoremstyle{definition}
\newtheorem{remark}[theorem]{Remark}

\newenvironment{keywords}{\bigskip\textbf{Keywords.}\enspace}{}

\def\beq{\begin{equation}}
\def\beql#1{\beq\label{#1}}
\def\eeq{\end{equation}}
\def\beqa{\beq\begin{aligned}}
\def\beqal#1{\beql{#1}\begin{aligned}}
\def\eeqa{\end{aligned}\eeq}
\def\bseq{\begin{subequations}}
\def\bseql#1{\begin{subequations}\label{#1}}
\def\eseq{\end{subequations}}

\def\deq#1{(\ref{#1})}

\def\bib#1{\bibitem{#1}}

\def\rb#1{\left(#1\right)}
\def\frb#1{\!\rb{#1}}
\def\brb#1{\big(#1\big)}

\def\sqb#1{\left[#1\right]}

\def\abs#1{\left|#1\right|}
\def\babs#1{\big|#1\big|}

\def\norm#1{\left\|#1\right\|}

\def\hpar#1{^{(#1)}}

\def\mat#1{\begin{matrix} #1 \end{matrix}}

\def\veps{\varepsilon}
\def\a{\alpha}
\def\G{\Gamma}
\def\g{\gamma}
\def\lam{\lambda}

\def\om{\omega}
\def\Del{\Delta}
\def\del{\delta}

\def\ka{\kappa}
\def\vphi{\varphi}

\def\ol{\overline}

\def\C{\mathbb{C}}
\def\R{\mathbb{R}}

\def\cB{\mathcal{B}}

\def\cI{\mathcal{I}}

\def\sbt{\subset}

\def\dif{\operatorname{d}\!}

\def\re{\operatorname{Re}}

\def\me{\mathrm{e}}

\def\diag{\operatorname{diag}}

\graphicspath{{./figures/}}

\def\sfw{0.49\linewidth}
\def\ssfw{1.0\linewidth}
\def\sSubref#1{(\subref*{#1})~}

\def\delt{h}
\def\tref{{\mathrm{ref}}}

\def\One{\mathbf{1}}
\def\One{\boldsymbol{1}}

\def\Gf#1{\G\frb{#1}}

\begin{document}

\title{A Gauss-Jacobi Kernel Compression Scheme for Fractional Differential Equations%
	\thanks{This is a version of an article published in Journal of Scientific Computing.
	The final publication is available online at http://dx.doi.org/10.1007/s10915-018-0848-x}}
\author{Daniel Baffet%
		\thanks{Department of Mathematics and Computer Science,
		University of Basel,
		Spiegelgasse 1, 4051 Basel, Switzerland}} 
\date{\today}

\maketitle
\thispagestyle{empty}
\markboth{{D.\ Baffet}}
	{{A Gauss-Jacobi Kernel Compression Scheme}}


\begin{abstract}
A scheme for approximating the kernel $w$ of the fractional $\a$-integral by a linear combination of exponentials is proposed and studied.
The scheme is based on the application of a composite Gauss-Jacobi quadrature rule to an integral representation of $w$.
This results in an approximation of $w$ in an interval $[\del,T]$, with $0<\del$, which converges rapidly in the number $J$ of quadrature nodes associated with each interval of the composite rule.
Using error analysis for Gauss-Jacobi quadratures for analytic functions, an estimate of the relative pointwise error is obtained.
The estimate shows that the number of terms required for the approximation to satisfy a prescribed error tolerance is bounded for all $\a\in(0,1)$, and that $J$ is bounded for $\a\in(0,1)$, $T>0$, and $\del\in(0,T)$.

\end{abstract}

\section{Introduction}

The nonlocal nature of the fractional integral and the singularity of its kernel make the numerical treatment of fractional differential equations (FDEs) considerably more difficult than that of standard differential equations.
Direct approaches (e.g., \cite{PredCorr,YA,jin,brun} and the references therein) for discretizing FDEs involving a fractional operator with respect to time require that the entire solution history is stored and used throughout the computation.
This can be expensive in terms of both computational and memory costs.

Several approaches for reducing these costs have been proposed.
In particular, methods based on the approximation of the kernel $w$ of the fractional $\a$-integral by a linear combination of exponentials have been proven effective.
See for example \cite{FastConvNRBC,OblConvQuad,AdaptOblConv,ErrEstTrpInvLap,AppExpSum}, and \cite{jrl}.
In this approach, the fractional integral $\cI^\a f\frb{t+\del}=w\ast f\frb{t+\del}$ of a function $f$ at $t+\del$, with $t\ge0$ and $\del>0$, is split into a local term and a history term of the form $w_\del\ast f\frb{t}$ with $w_\del\frb{t}=w\frb{t+\del}$.
Thus the approximation to the history term $w_\del\ast f$ of the fractional integral is given by the convolution of $f$ and $S$ of the form
\beql{eq:expSum}
	S\frb{t}=\sum_{p=1}^P b_p\me^{-a_p t}\ ,
\eeq
where, in general, $a_p$ and $b_p$, with $p=1,\ldots,P$, are complex numbers.
This approach has two main advantages.
The first is that the convolution $S\ast f$ requires only local information of $f$ to advance.
This property is discussed further in Section \ref{sec:overview}.
The second strength of the approach is that $w$ may be approximated very well at a positive distance $\del$ from its singularity by $S$ of the form \deq{eq:expSum} with a relatively small number of terms.
Approximations are usually derived by applying a quadrature rule to an integral representation of $w$; many relying on the identity
\beql{eq:IntrowIntRep}
	w\frb{t}=\frac{t^{-1+\a}}{\Gf{\a}}=\frac{1}{\Gf{\a}\Gf{1-\a}}\int_0^\infty s^{-\a}\me^{-ts}\dif s
\eeq
for a starting point, where $\Gf{\cdot}$ denotes the gamma function.
A brief qualitative discussion comparing some methods is presented in \cite{KC_FDEs}.
More recently methods have been proposed based on a Gauss-Laguerre quadrature \cite{ZTB} and the application of the trapezoidal rule \cite{McLean} to an integral representation obtained  from \deq{eq:IntrowIntRep} by substituting the integration variable.

In this paper we propose a scheme based on the application of a composite Gauss-Jacobi quadrature to \deq{eq:IntrowIntRep}.
We divide the integration interval $(0,\infty)$ in \deq{eq:IntrowIntRep} into $K+1$ finite intervals, and an infinite interval.
By neglecting the integral over the infinite interval and approximating the integral over each of the remaining $K+1$ finite intervals by an appropriate Gauss-Jacobi quadrature with $J$ quadrature nodes, we obtain an approximation $S$ of $w_\del$ in $[0,T-\del]$.
This yields an approximation of the form \deq{eq:expSum} with real parameters $a_1,\ldots,a_P$, and $b_1,\ldots,b_P$, and $P=(K+1)J$.

The application of Gaussian quadratures for approximating \deq{eq:IntrowIntRep} seems natural due to the rapid convergence they provide for analytic functions, and since they yield real parameters $a_p$ and $b_p$, with $p=1,\ldots,P$.
Indeed, schemes based Gaussian quadratures have been explored, e.g., in \cite{jrl} and the references therein, \cite{FastCaputo} and \cite{ZTB}.
The scheme of \cite{jrl} is obtained by applying a composite Gauss quadrature to \deq{eq:IntrowIntRep} after a substitution of the integration variable.
The resulting approximation $S$ satisfies $|w\frb{t}-S\frb{t}|<\veps$, for $t\in[\del,\infty)$, with
\[
	P=O\frb{\rb{1-\a}^{-1}\rb{\log\rb{\a\veps}^{-1}+\log\del^{-1}}^2}\ .
\]
The scheme of \cite{FastCaputo} is also obtained by applying a composite Gauss-Jacobi to \deq{eq:IntrowIntRep}.
However, the latter scheme and the present one differ in the choice of the intervals of the composite rule.

The scheme proposed in this paper is inspired by \cite{KC_FDEs}.
The latter is based on a multipole approximation of the Laplace transform of $w_\del$.
This scheme has two main drawbacks compared to schemes based on Gaussian quadratures.
One drawback is the convergence rate of the quadrature rule.
The term corresponding the number $J$ of quadrature nodes associated with each interval in the composite integration rule decays as $3^{-J}$.
Another drawback of the scheme of \cite{KC_FDEs} is that it yields complex parameters $a_1,\ldots,a_P$ and $b_1,\ldots,b_P$.
This scheme, however, also has some useful properties:
the number $P$ of terms required to satisfy an error tolerance is bounded for $\a\in(0,1)$, the number $J$ of quadrature nodes in each interval of the composite rule is independent of $\a$, $\del$ and $T$, and the approximation has a modular structure which makes it convenient for use within adaptive step-size schemes \cite{KC_Adapt}.

The present scheme retains these properties, while benefits from the rapid convergence of a Gauss-Jacobi quadrature.
For the scheme of \cite{KC_FDEs}, we have an estimate of the form
\beql{Intro:RelL2Est}
	\norm{\rb{w_\del-S}\ast f}_{L^2(0,T)} \le \veps \norm{w_\del\ast f}_{L^2(0,T)}\ ,
\eeq
however here we do not attempt to derive a similar result.
The main result of this paper, given by Theorem \ref{thm:main}, provides an estimate of the relative pointwise error of the approximation.
More precisely, for $\a\in(0,1)$, and $0<\del<T$, the approximation \deq{eq:expSum} prescribed by scheme, with $P=(K+1)J$, $K\ge 0$ and $J\ge1$, satisfies
\beq
	|w_\del\frb{t}-S\frb{t}| \le C_{KJ}\frb{\a,\del T^{-1}}\, w_\del\frb{t}
\eeq
for $t\in[0,T-\del]$, where
\beq
	C_{KJ}\frb{\a,\eta}=CA_J+B_K\frb{\a,\eta}
\eeq
with $C>0$ independent of the parameters of the problem, and
\beq
	A_J=J\rb{3+\sqrt{8}}^{-2J}\ ,
	\qquad
	B_K\frb{\a,\eta}=\frac{\G\brb{1-\a,\eta 2^K}}{\Gf{1-\a}}\ ,
\eeq
where $\Gf{\cdot,\cdot}$ is the upper incomplete gamma function.
By estimating the upper incomplete gamma function, we have that for $\a\in(0,1)$, and an error tolerance $\veps$, there holds $|w_\del-S|\le \veps w_\del$ in $[0,T-\del]$, for
\beql{Intropm}
	K=O\frb{\log\del^{-1}T+\log\log\frac{1-\a}{\veps}}\ ,
	\qquad
	J =O\frb{\log\veps^{-1}}\ .
\eeq
To obtain our estimates we employ error estimates for Gauss-Jacobi quadratures for functions analytic in a circle.
This approach differers from the approaches of \cite{jrl} and \cite{FastCaputo}, however may be applied to analyze their constructions.

The rest of the paper is structured as follows.
In Section \ref{sec:overview} we present an overview of the method, discuss its incorporation into time stepping schemes and introduce some notation.
We discuss error estimates for Gauss-Jacobi quadratures for functions analytic in a circle in Section \ref{sec:GJQ}.
The approximation is stated explicitly in Section \ref{sec:KerComp}, where the main result of the paper, given by Theorem \ref{thm:main}, concerning the kernel compression scheme and its proof are presented.
Details on the time stepping schemes used are provided in Section \ref{sec:SteppScheme}, and numerical results are provided in Section \ref{sec:NumRes}.
We conclude with some remarks in Section \ref{sec:Conc}.

\section{Overview}\label{sec:overview}

In the following we discuss the main idea of kernel compression schemes, propose an approach for incorporating such a scheme into a fully discrete time stepping method and introduce the notation used subsequently.
For $\a>0$, let
\beq
	\cI^\a f\frb{t}=\frac{1}{\Gf{\a}}\int_0^{t} (t-s)^{-1+\a} f\frb{s}\dif s
\eeq
be the fractional $\a$-integral of $f$.
We split $\cI^\a f\frb{t+\del}$ into a local term
\beq
	\int_0^{\del} w\frb{\del-s} f\frb{t+s}\dif s
\eeq
and a history term
\beql{fracHist}
	w_\del\ast f\frb{t}=\int_0^{t} w_\del\frb{t-s} f\frb{s}\dif s\ ,
\eeq
where $t\ge0$, $\del>0$ and
\beq
	w_\del\frb{t}=w\frb{t+\del}\ ,
	\qquad
	w\frb{t}=\frac{t^{-1+\a}}{\Gf{\a}}\ .
\eeq

\subsection{Kernel compression}

Let $\a\in(0,1)$, $f:(0,T)\to\R^d$, and $\del\in(0,T)$.
Consider the history term \deq{fracHist} of $\cI^\a f\frb{t+\del}$, the fractional integral of $f$ at $t+\del$.
As \deq{fracHist} has the form of a Laplace convolution of $f$ with the kernel $w_\del$, we seek an approximation
\beq
	S\frb{t}=\sum_{p=1}^P b_p\me^{-a_pt}
\eeq
to $w_\del$.
Formally substituting $w_\del$ by $S$ in the convolution we obtain
\beq
	S\ast f\frb{t}=\sum_{p=1}^P b_p \int_0^t \me^{-a_p(t-s)}f\frb{s}\dif s\ ,
\eeq
which may be expressed in terms of the solution to an initial value problem for a standard ODE system.
For each $p=1,\ldots,P$, define
\beq
	\psi_p\frb{t}=\int_0^t \me^{-a_p (t-s)}f\frb{s}\dif s\ .
\eeq
To simplify notation we organize $\psi_1,\ldots,\psi_P$ as the columns of a matrix $\Psi=(\psi_p)$.
Thus the approximation $S\ast f$ of $w_\del\ast f$ is given by
\beq
	S\ast f=\sum_{p=1}^P b_p\psi_p=\Psi b\ ,
\eeq
where $b=(b_1,\ldots,b_P)^T$.
Observing that each $\psi_p$ is the solution to the ODE $\psi'=-a_p\psi+f$ satisfying $\psi\frb{0}=0$, we recover
\beql{auxODE}
	\Psi'=-\Psi A+f\One\ ,
	\qquad
	\Psi\frb{0}=0\ ,
\eeq
where $\One=(1,\ldots,1)$ is a row $P$-vector, and $A=\diag\frb{a_1,\ldots,a_P}$.

\subsection{Incorporation into a time stepping scheme}
\label{sec:TimeStepSchemeSetup}

While the main focus of this paper is the approximation of \deq{fracHist}, the goal of the method is to be incorporated into fully discrete time-stepping schemes.
In particular, schemes for initial value problems
\beql{IVPgen}
	D^\a u=f\frb{t,u}\ ,
	\qquad
	u\frb{0}=u_0
\eeq
in $(0,T)$, where $\a\in(0,1)$, $f:[0,T]\times\Pi\to\R^d$, $T>0$, $\Pi\sbt\R^d$ open, and $D^\a$ the Caputo $\a$-derivative, given by
\beq
	D^\a u=\cI^{1-\a} u' \ .
\eeq
Below is a setup for the application of the kernel compression scheme as a part of a fully discrete time-stepping method for \deq{IVPgen}.

Application of $\cI^\a$ to \deq{IVPgen} yields
\beql{voltEq}
	u=u_0+\cI^\a\frb{f\circ u}\ ,
\eeq
where $f\circ u\frb{t}=f\frb{t,u\frb{t}}$.
In fact, \deq{IVPgen} and \deq{voltEq} are equivalent, provided $f$ is continuous \cite{DietFord}.
If, in addition, $f$ is Lipschitz in $u$, then \deq{voltEq} has a unique solution in a neighborhood of $t=0$.
In the following we assume \deq{voltEq} has a unique solution in $[0,T]$.
A standard approach for the derivation of numerical methods for \deq{voltEq}, and thus for \deq{IVPgen}, is as follows.
Fix $t\ge0$ and $\delt>0$, and let $\del\in(0,\delt)$.
Owing to \deq{voltEq}, we have
\bseql{voltEquloc}
\begin{align}
	\label{uloc}
	u\frb{t+\del} &=\int_0^{\del} w\frb{\del-s} f\frb{t+s,u\frb{t+s}}\dif s
		+H\frb{t,\del} \\
	H\frb{t,\del} &=u_0+w_\del\ast\frb{f\circ u}\frb{t}\ .
\end{align}
\eseq
Observing that equation \deq{uloc} has the form
\beql{Uloc}
	U\frb{\del}=\cI^\a\frb{F\circ U}\frb{\del}+H\frb{\del}\ ,
\eeq
where $U\frb{\del}=u\frb{t+\del}$, $F\frb{\del,u}=f\frb{t+\del,u}$ and $H\frb{\cdot}=H\frb{t,\cdot}$, we find that two ingredients are required for the time-stepping scheme.
The first ingredient is a method for approximating Volterra equations \deq{Uloc} on short intervals $(0,\delt)$, assuming $H\frb{\cdot}$ is given.
This scheme is applied to \deq{uloc} to advance the numerical solution from $t$ to $t+\delt$, and perhaps compute approximations of $u$ at a small number of points in $(t,t+\delt)$.
In practice, schemes approximating \deq{Uloc} require $H\frb{\cdot}=H\frb{t,\cdot}$ to be evaluated at some points in $(0,\delt]$.
This requires computing the history term, and is therefore expensive to perform.
To reduce the costs of evaluating the history term, we require the second ingredient of the scheme -- the kernel compression scheme.
This scheme prescribes the approximation in terms of the matrix of auxiliary variables $\Psi$, defined as the solution to \deq{auxODE}.
Thus we also need a method for approximating \deq{auxODE}.
For that purpose, we may use an A-stable method.
The reason an A-stable method is required is that some of the $a_p$-s are large and positive.
The time stepping schemes tested in this paper are discussed in Section \ref{sec:SteppScheme}.

\section{Error estimates for Gauss-Jacobi quadratures}
\label{sec:GJQ}

Let $w\hpar{a,b}\frb{x}=(1-x)^a(1+x)^b$ with $a,b>-1$, and
\beql{eq:GJQuad}
	\int_{-1}^1 f\frb{x}w\hpar{a,b}\frb{x} \dif x =\sum_{k=1}^n \om_kf\frb{\xi_k} +E\hpar{a,b}_n\frb{f}
\eeq
the Gauss-Jacobi integration rule associated with the weight $w\hpar{a,b}$, where $E\hpar{a,b}_n\frb{f}$ is the error.
If $f$ is analytic in an open set containing the closure $\ol{\cB\frb{\ell}}$ of  the disc
\beq
	\cB\frb{\ell}=\{z\in\C\, :\, |z|<\ell\}
\eeq
with $\ell>1$, then \cite{GautschiVarga} the error $E\hpar{a,b}_n\frb{f}$ is given by
\beq
	E\hpar{a,b}_n\frb{f}=\frac{1}{2\pi i}\int_{|z|=\ell} K\hpar{a,b}_n\frb{z}f\frb{z}\dif z\ ,
\eeq
where
\beq
	K\hpar{a,b}_n=\frac{\Pi_n\hpar{a,b}}{P_n\hpar{a,b}}\ ,
\eeq
with $P_n\hpar{a,b}$ the Jacobi polynomial of degree $n$ normalized so that
\beq
	P_n\hpar{a,b}\frb{1}=\rb{\mat{n+a\\ n}}\ ,
\eeq
and $\Pi_n\hpar{a,b}$ given by
\beq
	\Pi_n\hpar{a,b}\frb{z}=\int_{-1}^1\frac{P_n\hpar{a,b}\frb{x}}{z-x}\, w\hpar{a,b}\frb{x}\dif x\ ,
	\qquad
	z\in\C\setminus[-1,1]\ .
\eeq
The function $\Pi_n\hpar{a,b}$ is related to the the Jacobi function of the second kind $Q\hpar{a,b}_n$ through \cite{transFunc2}
\beql{EqPiQ}
	\Pi\hpar{a,b}_n\frb{x}=2w\hpar{a,b}\frb{x} Q\hpar{a,b}_n\frb{x}\ .
\eeq
Owing to (28) of Section 10.7 of \cite{transFunc2}, for $z\in\C\setminus[-1,1]$, there holds
\beql{PI_IntRep}
	\Pi\hpar{a,b}_n\frb{z}
		=2^{-n}\int_{-1}^1\frac{\rb{1-x}^{n+a}\rb{1+x}^{n+b}}{\rb{z-x}^{n+1}} \dif x\ .
\eeq
In particular, for $\ell>1$, this yields $K_n\hpar{a,b}\frb{\ell}>0$.
Due to basic properties of $P_n\hpar{a,b}$, and \deq{PI_IntRep}, we obtain $K_n\hpar{a,b}\frb{z}=-K_n\hpar{b,a}\frb{-z}$.
This combined with Theorem 3.1 of \cite{GautschiVarga} yield
\beq
	\max_{|z|=\ell} \abs{K\hpar{a,b}_n\frb{z}}
		=\begin{cases}
			K\hpar{a,b}_n\frb{\ell} & b\ge a \\[0.5em]
			K\hpar{b,a}_n\frb{\ell} & a\ge b
		\end{cases}\ .
\eeq
As a result we have the following:

\begin{corollary}\label{Cor:GJ_ErrEst}
Let $a,b>-1$, $A=\min(a,b)$, and $B=\max(a,b)$.
If $f$ is analytic in an open set containing the closure $\ol{\cB\frb{\ell}}$ of $\cB\frb{\ell}$, with $\ell>1$, then the error $E\hpar{a,b}_n\frb{f}$ of the Gauss-Jacobi quadrature rule \deq{eq:GJQuad} satisfies
\beq
	\abs{E\hpar{a,b}_n\frb{f}} \le \ell K\hpar{A,B}_n\frb{\ell} \max_{|z|=\ell} \abs{f\frb{z}}\ .
\eeq
\end{corollary}

To estimate the error, we may estimate $K\hpar{a,b}_n\frb{\ell}$ for $\ell>1$.
The estimate below is based on asymptotic expansions derived in \cite{UnifAsymptJacobi}.

\begin{theorem}\label{thm:GJ_ErrKern}
Let $\ell_\ast>1$, and $D\sbt[-1,\infty)^2$ compact.
Then, at the limit $n\to\infty$, there holds
\beql{eq:KAsympEquiv}
	K\hpar{a,b}_n\frb{\ell} \sim 2\pi\
		\frac{(\ell-1)^{a}(\ell+1)^{b}}{\rb{\ell+\sqrt{\ell^2-1}}^{a+b}}\, \rb{\ell+\sqrt{\ell^2-1}}^{-(2n+1)}
\eeq
uniformly with respect to $(a,b,\ell)\in D\times[\ell_\ast,\infty)$.
Moreover, there exists a constant $C>0$ such that the estimate
\beql{eq:KabEst}
	K\hpar{a,b}_n\frb{\ell} \le
		C \rb{\ell+\sqrt{\ell^2-1}}^{-(2n+1)}
\eeq
holds for every $n\ge 2$, and $(a,b,\ell)\in D\times[\ell_\ast,\infty)$.
If, in addition $(-1,-1)\notin D$, then there exists a constant $C$, such that \deq{eq:KabEst} holds for every $n\ge1$ and $(a,b,\ell)\in D\times[\ell_\ast,\infty)$.
\end{theorem}

\begin{remark}
Estimates \deq{eq:KAsympEquiv} and \deq{eq:KabEst} are obtained by manipulating the leading terms of asymptotic expansions \cite{UnifAsymptJacobi} of $P\hpar{a,b}_n\frb{\ell}$ and $\Pi\hpar{a,b}_n\frb{\ell}$.
The latter two capture the asymptotic behavior of their respective functions for $\ell$ fixed and large $n$, and for $n$ fixed and large $\ell$ and thus may yield estimates of $K\hpar{a,b}_n\frb{\ell}$ that are, in some cases, more accurate than \deq{eq:KAsympEquiv} and \deq{eq:KabEst}.
However, in this paper, we only require estimates of $K\hpar{a,b}_n\frb{\ell}$ with $\ell\approx 3$, and \deq{eq:KabEst} suits that purpose.
\end{remark}

\begin{proof}
We recall the leading term approximations of the asymptotic expansions \cite{UnifAsymptJacobi} of $P\hpar{a,b}_n$ and $\Pi\hpar{a,b}_n$.
For $a,b>-1$, at the limit $n\to\infty$, there hold
\beql{eq:PAsympEquiv0}
	P\hpar{a,b}_n\frb{\ell} \sim \rho_n\frb{a,b}
		\frac{\rb{\ell+\sqrt{\ell^2-1}}^{n+\ka/2}}{(\ell-1)^{(2a+1)/4}(\ell+1)^{(2b+1)/4}}
\eeq
and
\beql{eq:PiAsympEquiv0}
	\Pi\hpar{a,b}_n\frb{\ell} \sim \pi_n\frb{a,b}
		\frac{(\ell-1)^{(2a-1)/4}(\ell+1)^{(2b-1)/4}}{\rb{\ell+\sqrt{\ell^2-1}}^{n+\ka/2}}
\eeq
uniformly with respect to $\ell\in[\ell_\ast,\infty)$, where $\ka=a+b+1$ (notice the different notation),
\beq
	\rho_n\frb{a,b} = \frac{\Gf{2n+\ka}}{2^{2n+\ka/2} \Gf{n+1}\Gf{n+\ka}}\ ,
\eeq
and
\beq
	\pi_n\frb{a,b} =\frac{2^{2n+3\ka/2} \Gf{n+a+1}\Gf{n+b+1}}{\Gf{2n+a+b+2}}\ .
\eeq
These results are obtained in \cite{UnifAsymptJacobi} by transforming the differential equation satisfied by $P\hpar{a,b}_n$ and $Q\hpar{a,b}_n$ into an equation of the form
\[
	W''\frb{z}=\rb{\rb{2n+\ka}^2+F\frb{a,b,z}}W\frb{z}\ ,
\]
and applying Theorem A of \cite{Olver1}.
In fact, a similar argument relying on Theorem A of \cite{Olver2}, instead of Theorem A of \cite{Olver1}, yields that for $n\to\infty$, each of the asymptotic equivalences \deq{eq:PAsympEquiv0} and \deq{eq:PiAsympEquiv0} holds uniformly with respect to $(a,b,\ell)\in D\times[\ell_\ast,\infty)$.
Therefore, we have
\beql{eq:KAsympEquiv0}
	K\hpar{a,b}_n\frb{\ell}\sim \frac{\pi_n\frb{a,b}}{\rho_n\frb{a,b}}\,
		\frac{(\ell-1)^{a}(\ell+1)^{b}}{\rb{\ell+\sqrt{\ell^2-1}}^{a+b}}\,
		\rb{\ell+\sqrt{\ell^2-1}}^{-(2n+1)}
\eeq
uniformly with respect to $(a,b,\ell)\in D\times[\ell_\ast,\infty)$, at the limit $n\to\infty$.
To simplify the right hand side of \deq{eq:KAsympEquiv0}, we use that, by Stirling's approximation, there holds
\beq
	\frac{\Gf{x+a}\Gf{x+b}}{\Gf{2x+c}}
			\sim 2^{1-c}\sqrt{\pi}\, x^{a+b-c-1/2}\, 2^{-2x}\ ,
	\qquad
	x\to\infty\ ,
\eeq
uniformly for $(a,b,c)$ bounded, and therefore the limit
\beq
	\lim_{n\to\infty}\frac{\pi_n\frb{a,b}}{\rho_n\frb{a,b}} =2\pi
\eeq
exists uniformly in $D$.
As a result, we have \deq{eq:KAsympEquiv}, at the limit $n\to\infty$, uniformly in $D\times[\ell_\ast,\infty)$.

It is now left to show that there exists a constant $C>0$ such that \deq{eq:KabEst} holds for each $n\ge 2$, and $(a,b,\ell)\in D\times[\ell_\ast,\infty)$, and that if $(-1,-1)\notin D$, then the same holds true for $n=1$.
It is easy to verify that each of the asymptotic equivalences \deq{eq:PAsympEquiv0} with $n\ge2$, and \deq{eq:PiAsympEquiv0} with $n\ge1$, holds at the limit $\ell\to\infty$ uniformly with respect to $(a,b)\in D$.
The restriction $(-1,-1)\notin D$ is sufficient to ensure that the above is also true of \deq{eq:PAsympEquiv0} with $n=1$.
Therefore, the equivalence \deq{eq:KAsympEquiv0} also holds for each $n\ge 2$ (and for $n=1$, if $(-1,-1)\notin D$) at the limit $\ell\to\infty$, uniformly in $D$.
In particular, this yields that for each $n$, the function
\beq
	(a,b,\ell)\longmapsto K\hpar{a,b}_n\frb{\ell} \rb{\ell+\sqrt{\ell^2-1}}^{2n+1}
\eeq
is bounded in $D\times[\ell_\ast,\infty)$.
This, combined with \deq{eq:KAsympEquiv}, which holds uniformly with respect to $(a,b,\ell)\in D\times[\ell_\ast,\infty)$, at the limit $n\to\infty$, yield the conclusion.
\end{proof}

\section{Gauss-Jacobi kernel compression}
\label{sec:KerComp}

Let $\a\in(0,1)$, and $T>\del>0$.
To approximate the kernel $w_\del$ of the history term of the fractional integral in $[0,T-\del]$ by a sum of exponentials, we apply a composite Gauss-Jacobi quadrature to the following integral representation of $w$
\beql{wIntRep}
	w\frb{t}=\frac{t^{-1+\a}}{\Gf{\a}}=C\frb{\a}\int_0^\infty s^{-\a}\me^{-ts}\dif s\ ,
\eeq
for $t\in[\del,T]$, where
\beq
	C\frb{\a}=\frac{1}{\Gf{\a}\Gf{1-\a}}=\frac{\sin\frb{\pi\a}}{\pi}\ .
\eeq
That is, we choose non-overlapping intervals
\beql{eq:Intervals}
	\brb{c_{k}-r_k,c_k+r_k}\ ,
	\qquad
	k=0,\ldots,K\ ,
\eeq
split the integral on the right hand side of \deq{wIntRep} into a sum
\beql{wIntRep_Split}
	w\frb{t}= \sum_{k=0}^K \int_{c_{k}-r_{k}}^{c_{k}+r_{k}} +\int_{c_K+r_K}^\infty\ ,
\eeq
neglect the integral over the infinite interval, and approximate each of the remaining terms by an appropriate Gauss-Jacobi quadrature.
Thus, we obtain an approximation $w\frb{t}\approx S\frb{t-\del}$, for $t\in[\del,T]$, where
\bseql{eq:wdelApprox}
\beq
	S\frb{t}=\sum_{k=0}^K\sum_{j=1}^J b_{kj}\me^{-a_{kj}t}\ ,
\eeq
with
\beq
	a_{0j} =r_{0}\xi\hpar{0,-\a}_j +c_0\ ,
	\quad
	b_{0j} =C\frb{\a}\me^{-\del a_{0j}} r_{0}^{1-\a}\om\hpar{0,-\a}_j\ ,
	\qquad
	j=1,\ldots,J\ ,
\eeq
and for each $k=1,\ldots,K$,
\beq
	a_{kj} =r_{k}\xi\hpar{0,0}_j+c_k\ ,
	\quad
	b_{kj} =C\frb{\a} \me^{-\del a_{kj}} a_{kj}^{-\a} r_{k}\om\hpar{0,0}_j\ ,
	\qquad
	j=1,\ldots,J\ .
\eeq
Here $\xi\hpar{a,b}_1,\ldots,\xi\hpar{a,b}_J$ and $\om\hpar{a,b}_1,\ldots,\om\hpar{a,b}_J$ denote the Gauss-Jacobi quadrature nodes and weights, respectively, associated with $w\hpar{a,b}\frb{x}=(1-x)^{a}(1+x)^{b}$.
The centers $c_1,\ldots,c_K$ and radii $r_1,\ldots,r_K$ of the intervals \deq{eq:Intervals} are chosen so the resulting errors associated with the different intervals converge at similar rates.
This requirement yields
\beq
	r_0=c_0=\frac{1}{2T}\ ,
\eeq
and
\beq
	r_k=\frac{1}{2T} 2^{k-1}\ ,
	\quad
	c_k=3r_k=\frac{3}{2T}\, 2^{k-1}\ ,
	\qquad
	k=1,\ldots,K\ .
\eeq
\eseq

In the following, $C$, $C_0$, $C_1$ and so on, denote generic constants that are independent of the parameters of the problem, and may have different values in different places.
For the approximation above we have the following result.

\begin{theorem}\label{thm:main}
Let $\a\in(0,1)$, $T>0$, $\del\in(0,T)$, and $S$ given by \deq{eq:wdelApprox}.
There exists a constant $C>0$, independent of $\a$, $\del$ or $T$, such that for each $t\in[0,T-\del]$, the following estimate holds
\bseql{eq:mainErrEst}
\beq
	\abs{w\frb{t+\del}-S\frb{t}}\le C_{KJ}\frb{\a, \del T^{-1}}\, w\frb{t+\del}\ ,
\eeq
where
\beq
	C_{KJ}\frb{\a,\eta}=CA_J+B_K\frb{\a,\eta}
\eeq
with
\beql{eq:AJ_BK}
	A_J=J\rb{3+\sqrt{8}}^{-2J}\ ,
	\qquad
	B_K\frb{\a,\eta}=\frac{\G\brb{1-\a ,\eta 2^K}}{\Gf{1-\a}}\ .
\eeq
\eseq
\end{theorem}

The proof of Theorem \ref{thm:main} is at the end of this section.
To simplify the analysis, the following two lemmas treat the case where $\del=1$.
The general result, given by Theorem \ref{thm:main}, is obtained as a corollary by rescaling the problem.

\begin{lemma}\label{lem:W1approx}
There exists a constant $C>0$ such that for each $\a\in(0,1)$, $T>1$, $K\ge 0$, $J\ge 1$, and $t\in[1,T]$, $S$ given by \deq{eq:wdelApprox} with $\del=1$ satisfies
\beql{eq:W1_est}
	\babs{W_1\frb{t}-S\frb{t-1}}\le C A_J\, w\frb{t}\ ,
\eeq
where $A_J$ is given by \deq{eq:AJ_BK}, and
\beql{eq:W1}
	W_1\frb{t}=C\frb{\a}\int_0^{2^{K}/T} s^{-\a}\me^{-ts}\dif s\ .
\eeq
\end{lemma}

\begin{proof}
Let $t\in[1,T]$, and
\beq
	W_1\frb{t}=\int_0^{2^K/T} \vphi\frb{t,s}\dif s\ ,
	\qquad
	\vphi\frb{t,s}=C\frb{\a}s^{-\a}\me^{ts}\ .
\eeq
The approximation $S\frb{t-1}$ of $W_1\frb{t}$ is obtained by dividing the interval $[0,2^K/T]$ into non-overlapping intervals $(c_k-r_k,c_k+r_k)$, with $k=0,\ldots,K$, and approximating the integral of $\vphi\frb{t,\cdot}$ over each interval by the appropriate Gauss-Jacobi quadrature.
Then, for each $k=0,\ldots,K$, we may use Corollary \ref{Cor:GJ_ErrEst} and estimate the error
\beql{Delk}
	\Del_{kJ}\frb{t}=\int_{c_k-r_k}^{c_k+r_k} \vphi\frb{t,s}\dif s
		-\sum_{j=1}^J b_{kj}\me^{-a_{kj}(t-1)}
\eeq
in the approximation of $\int_{c_k-r_k}^{c_k+r_k} \vphi\frb{t,s}\dif s$.

Let $k\ge 1$.
By substituting the integration variable $s=r_kx+c_k$, and using $c_k=3 r_k$, we obtain
\beq
			\int_{c_k-r_k}^{c_k+r_k} s^{-\a}\me^{-t s}\dif s
		= r_k^{1-\a}
			\int_{-1}^{1} \rb{3+x}^{-\a}\me^{-r_kt(3+ x)}\dif x\ .
\eeq
As $f_t\frb{z}=\rb{3+z}^{-\a}\me^{-r_kt(3+z)}$ is analytic for $|z|<3$, Corollary \ref{Cor:GJ_ErrEst} with $f=f_t$, and $a=b=0$ yields
\beq
	|\Del_{kJ}\frb{t}|\le 
		\ell K\hpar{0,0}_J\frb{\ell}\, C\frb{\a} r_k^{1-\a}\max_{|z|=\ell} \abs{f_t\frb{z}}
\eeq
for each $\ell\in(1,3)$.
Since for $z\in\C$ with $|z|=\ell$ there hold
\beq
	|3+z|\ge 3 -|z| = 3-\ell
\eeq
and
\beq
	\abs{\me^{-r_kt(3+z)}}=\me^{-r_kt(3+\re z)} \le \me^{-r_kt(3-\ell)}\ ,
\eeq
we get
\beq
	\abs{\Del_{kJ}\frb{t}} \le 
			\frac{\ell K\hpar{0,0}_J\frb{\ell}}{(3-\ell)^\a}\,
			C\frb{\a} r_k^{1-\a}\ \me^{-(3-\ell)tc_k}\ .
\eeq
As $c_k=3r_k$,
\[
	r_k=\frac{2}{3}\rb{c_k-c_{k-1}}\ ,
\]
and $\ell\in(1,3)$, we may find a constant $C_1$ independent of the parameters of the problem, such that the estimate
\beql{eq:Delk_est0}
	\abs{\Del_{kJ}\frb{t}} \le C_1K\hpar{0,0}_J\frb{\ell}\, q^{-\a}\vphi\frb{qt,c_k}\Del c_k
\eeq
holds for every $\ell\in(1,3)$, where $\Del c_k=c_k-c_{k-1}$, and $q=3-\ell$.
Thus, using Theorem \ref{thm:GJ_ErrKern} with $a=b=0$, we may find a constant $C_2$, independent of the parameters of the problem, such that for each $\ell\in[3/2,3)$, there holds
\beql{eq:Delk_est}
	\abs{\Del_{kJ}\frb{t}}\le C_2\, R_J\frb{\ell}\, q^{1-\a}\vphi\frb{qt,c_k}\Del c_k\ ,
\eeq
where
\beq
	R_J\frb{\ell}=(3-\ell)^{-1}\rb{\ell+\sqrt{\ell^2-1}}^{-2J}\ .
\eeq
The reason for introducing $R_J$ by multiplying and dividing the right hand side of \deq{eq:Delk_est0} by $q=3-\ell$ is made clear near the end of the proof.

Next we estimate the error of the approximation of the integral over $(c_0-r_0,c_0+r_0)$.
For $k=0$, we have $c_0=r_0=1/(2T)$, and therefore
\beq
	\int_{c_0-r_0}^{c_0+r_0} s^{-\a}\me^{-t s}\dif s
		= c_0^{1-\a} \int_{-1}^{1} \rb{1+x}^{-\a}\me^{-c_0t(1+x)}\dif x\ .
\eeq
The integral over $(-1, 1)$ is approximated by the Gauss-Jacobi quadrature associated with the weight function $\rb{1+x}^{-\a}$.
Since for $f_t\frb{z}=\me^{-c_0t(1+z)}$, and for each $\ell\in(1,3)$ there holds
\beq
	\max_{|z|=\ell} \abs{f_t\frb{z}}=\me^{-c_0t(1-\ell)}=\me^{2c_0t}\, \me^{-c_0t(3-\ell)}
		\le \me\, \me^{-c_0t(3-\ell)}\ ,
\eeq
by Corollary \ref{Cor:GJ_ErrEst} with $a=0$, $b=-\a$, $\ell\in(1,3)$ and $f=f_t\frb{z}=\me^{-c_0t z}$, there holds
\beq
	\abs{\Del_{0J}\frb{t}}\le \me \ell K_J\hpar{-\a,0}\frb{\ell}\, \vphi\frb{qt,c_0} \Del c_0\ ,
\eeq
where $q=3-\ell$,
\[
	c_0=c_0-0=\Del c_0\ .
\]
Note that since $f_t\frb{z}=\me^{-c_0t z}$ is an entire function, the requirement $\ell<3$ is not necessary for the application of Corollary \ref{Cor:GJ_ErrEst}.
That requirement, however, is necessary for the estimates below.
By Theorem \ref{thm:GJ_ErrKern} with $a=-\a$, and $b=0$, there exists a constant $C_0$, independent of the parameters of the problem, such that for all $\ell\in[3/2,3)$,
\beql{eq:Del0_est}
	\abs{\Del_{0J}\frb{t}}\le C_0\, R_J\frb{\ell}\, q^{1-\a}\vphi\frb{qt,c_0} \Del c_0\ .
\eeq

Now we may estimate the error of the approximation $S\frb{t-1}$ of $W_1\frb{t}$.
Combining estimates \deq{eq:Delk_est} and \deq{eq:Del0_est}, we recover
\beq
	\babs{W_1\frb{t}-S\frb{t-1}} \le \sum_{k=0}^K \abs{\Del_{kJ}\frb{t}}
		\le C\, R_J\frb{\ell}\, q^{1-\a} \sum_{k=0}^K \vphi\frb{qt,c_k}\Del c_k\ .
\eeq
As the sum on the right hand side of the inequality above is a lower Riemann sum, we get
\beq
	\babs{W_1\frb{t}-S\frb{t-1}}
		 \le C\, R_J\frb{\ell}\, q^{1-\a} \int_0^{c_K} \vphi\frb{qt,s}\dif s
		\le C\, R_J\frb{\ell}\, q^{1-\a} w\frb{qt}\ .
\eeq
Hence, the estimate
\beql{eq:W1_est1}
	\babs{W_1\frb{t}-S\frb{t-1}} \le C\, R_J\frb{\ell}\, w\frb{t}\ ,
\eeq
holds for every $\ell\in[3/2,3)$.
Since the left hand side of \deq{eq:W1_est1} is independent of $\ell$, we recover
\beql{eq:W1_est2}
	\babs{W_1\frb{t}-S\frb{t-1}} \le C \sqb{\inf_{\ell\in[3/2,3)} R_J\frb{\ell}}w\frb{t}\ .
\eeq
To complete the proof, we must show that the infimum on the right hand side of \deq{eq:W1_est2} satisfies an appropriate estimate at the limit where $J$ tends to infinity.
Such an estimate is provided by Lemma \ref{lem:estRJlJ}.
\end{proof}

The following lemma is the time-domain counterpart of estimate (3.22) of \cite{KC_FDEs}, and its proof is based on the same idea.
The lemma provides an estimate of the error of neglecting the integral over $(c_K+r_K,\infty)$ in \deq{wIntRep_Split}.

\begin{lemma}\label{lem:W2approx}
Let $\a\in(0,1)$, and
\beql{eq:W2}
	W_2\frb{a,t}=C\frb{\a}\int_a^{\infty} s^{-1+\a}\me^{-ts}\dif s\ .
\eeq
Then, for each $a\ge0$ and $t\ge1$, the following estimate holds
\beql{eq:W2Est}
	W_2\frb{a,t} \le \frac{\Gf{1-\a,a}}{\Gf{1-\a}} w\frb{t}\ .
\eeq
\end{lemma}

\begin{proof}
To prove \deq{eq:W2Est}, we show
\beq
	\int_a^\infty s^{-\a}\me^{-t s}\dif s \le \frac{\Gf{1-\a,a}}{\Gf{1-\a}} \int_0^\infty s^{-\a}\me^{-t s}\dif s\ .
\eeq
We do this by showing that
\beq
	\Del\frb{a,t} =\Gf{1-\a,a}\int_0^\infty s^{-\a}\me^{-t s}\dif s
			-\Gf{1-\a}\int_a^\infty s^{-\a}\me^{-t s}\dif s
\eeq
is nonnegative.
After cancelation of terms we get
\beq
	\Del\frb{a,t} =\Gf{1-\a,a}\int_0^a s^{-\a}\me^{-t s}\dif s -\g\frb{1-\a,a}\int_a^\infty s^{-\a}\me^{-t s}\dif s
\eeq
where $\g$ is the lower incomplete gamma function.
Thus, we have
\beqa
	\Del\frb{a,t}
		& =\Gf{1-\a,a}\int_0^a s^{-\a}\me^{-t s}\dif s -\g\frb{1-\a,a}\int_a^\infty s^{-\a}\me^{-t s}\dif s \\
		& =\int_0^a\int_a^\infty (su)^{-\a}\rb{\me^{-(t s+u)} -\me^{-(s+t u)}} \dif u\dif s
\eeqa
which, for $t\ge 1$, yields
\beq
	\Del\frb{a,t} =\int_0^a\int_a^\infty (su)^{-\a}\me^{-(t s+u)}
		\rb{1 -\me^{-(t-1)(u-s)}} \dif u\dif s \ge 0\ ,
\eeq
and hence the conclusion.
\end{proof}

We are now in a position to prove Theorem \ref{thm:main}.

\begin{proof}[Proof (Theorem \ref{thm:main})]
Let  $T_0>0$ and $\del_0\in(0,T_0)$.
We show that the theorem holds true for $\del=\del_0$ and $T=T_0$.
Let $S_0$ be of the form of $S$ given by \deq{eq:wdelApprox} with $\del=1$ and $T=\del_0^{-1}T_0$.
We have
\beq
	w\frb{\tau}=W_1\frb{\tau}+W_2\frb{2^K\del_0T_0^{-1},\tau}
\eeq
where $W_1$ and $W_2$ are given by \deq{eq:W1} and \deq{eq:W2}, respectively.
By Lemma \deq{lem:W1approx}, there exists a positive constant $C$ such that
\beq
	\abs{W_1\frb{\tau}-S_0\frb{\tau-1}} \le CA_Jw\frb{\tau}
\eeq
for each $\tau\in[1,\del_0^{-1}T_0]$.
Lemma \deq{lem:W2approx} states that
\beq
	W_2\frb{2^K\del_0T_0^{-1},\tau} \le B_K\frb{\a,\del_0T_0^{-1}} w\frb{\tau}
\eeq
for all $\tau\ge 1$.
Thus we recover
\beq
	\abs{w\frb{\tau}-S_0\frb{\tau-1}}\le \rb{CA_J+B_K\frb{\a,\del_0T_0^{-1}}} w\frb{\tau}
\eeq
for all $\tau\in[1,\del_0^{-1}T_0]$.
By multiplying the last inequality by $\del_0^{-1+\a}$ and substituting $\tau=\del_0^{-1}t+1$ we recover \deq{eq:mainErrEst} with $\del=\del_0$ and $T=T_0$, and $S\frb{t}=\del^{-1+\a}S_0\brb{\del_0^{-1}t}$.
Since the latter $S$ is of the form \deq{eq:wdelApprox} with $\del=\del_0$ and $T=T_0$, we have the conclusion.
\end{proof}

\section{Time stepping schemes}
\label{sec:SteppScheme}

A time stepping method for \deq{IVPgen} adopting the approach discussed in Section \ref{sec:TimeStepSchemeSetup} requires a scheme for approximating Volterra equations \deq{Uloc} on short intervals, a kernel compression scheme for approximating the history term, and a scheme for approximating \deq{auxODE}.
Below we test two methods of this type employing the kernel compression scheme \deq{eq:wdelApprox}.
The first method is given by
\bseql{eq:TR}
\beql{eq:VoltTR}
	v^{n+1}=\delt^\a\rb{W_0f^{n+1}+W_1f^{n}}+\Phi^n b
\eeq
and
\beql{auxODE_TR}
	\Phi^{n+1} =\sqb{\Phi^{n}\rb{I-\frac{\delt}{2}A} +\frac{1}{2}\rb{f^{n+1}+f^{n}}\One}
		\rb{I+\frac{\delt}{2}A}^{-1}
\eeq
\eseq
where $v^n$ and $\Phi^n$ are the numerical approximations of $u\frb{t_n}$ and $\Psi\frb{t_n}$, respectively, at the $n$-th time level $t_n$, $f^n=f\frb{t_n,v^n}$, and
\beq
	W_0=\frac{1}{\Gf{2+\a}}\ ,
	\qquad
	W_1=\frac{\a}{\Gf{2+\a}}\ .
\eeq
This is an implicit scheme which is obtained by formally replacing $F\circ U$ in \deq{Uloc} by an interpolating polynomial of degree one, and approximating \deq{auxODE} by the trapezoidal rule.
Note that since $A$ in \deq{auxODE_TR} is diagonal, advancing $\Phi$ does not require solving a large algebraic system of equations.
Advancing $v$, however, does require that we solve \deq{eq:VoltTR} for $v^{n+1}$.
In the tests below, this is done by Newton's method.
A similar method employing the kernel compression scheme of \cite{KC_FDEs} is tested therein.
The second method is a high order and adaptive method, denoted LER-IDC in \cite{KC_Adapt}, which we test here with the kernel compression scheme \deq{eq:wdelApprox}.
The method is obtained by applying an integral deferred correction scheme based on the left endpoint rule for the approximation of \deq{Uloc} and a 4th order, L-stable, diagonally implicit Runge-Kutta scheme for the approximation of \deq{auxODE}.
It employs adaptive step size control and modifies the kernel compression approximation accordingly.
For simplicity, we denote this scheme LER-IDR, similarly to \cite{KC_Adapt}.

\section{Numerical results}
\label{sec:NumRes}

In this section are results obtained with the kernel compression scheme.
The tests in Section \ref{sec:ApproxOKern} compare the kernel $w$ and the proposed kernel compression approximation directly, and do not involve a time stepping procedure.
The tests in Section \ref{sec:AppTimeSteppScheme} are obtained with the two fully discrete time stepping methods discussed in Section \ref{sec:SteppScheme}.

\subsection{Approximation of the kernel}
\label{sec:ApproxOKern}

To test estimate \deq{eq:mainErrEst}, we compare the relative error
\beql{relErr}
	\Del\frb{t}=\abs{\frac{w\frb{t}-S\frb{t-\del}}{w\frb{t}}}
\eeq
of the approximation of the kernel, where $S$ is given by \deq{eq:wdelApprox}, with $A_J$ and $B_K$.
In the following $P=(K+1)J$ is the total number of terms in $S$.

Figure \ref{fig:KerApprox_RelErr_t} shows the relative error $\Del$ as a function of $t$ in a neighborhood of $[\del,T]$, for two schemes:
Figure \ref{fig:KerApprox_RelErr_t_a} shows results obtained with the present scheme, and Figure \ref{fig:KerApprox_RelErr_t_b} shows results obtained with the scheme proposed in \cite{KC_FDEs}, where $p$ is the number of circles, $m$ the number of poles on each circle, and $P\approx pm/2$ the number of poles used in practice.
The results are obtained for $T=10^2$ and $\del=10^{-4}$.
\begin{figure}
\centering
\begin{subfigure}[t]{\sfw}%
	\centering
	\includegraphics[width=\ssfw,trim=0 0 30 0]{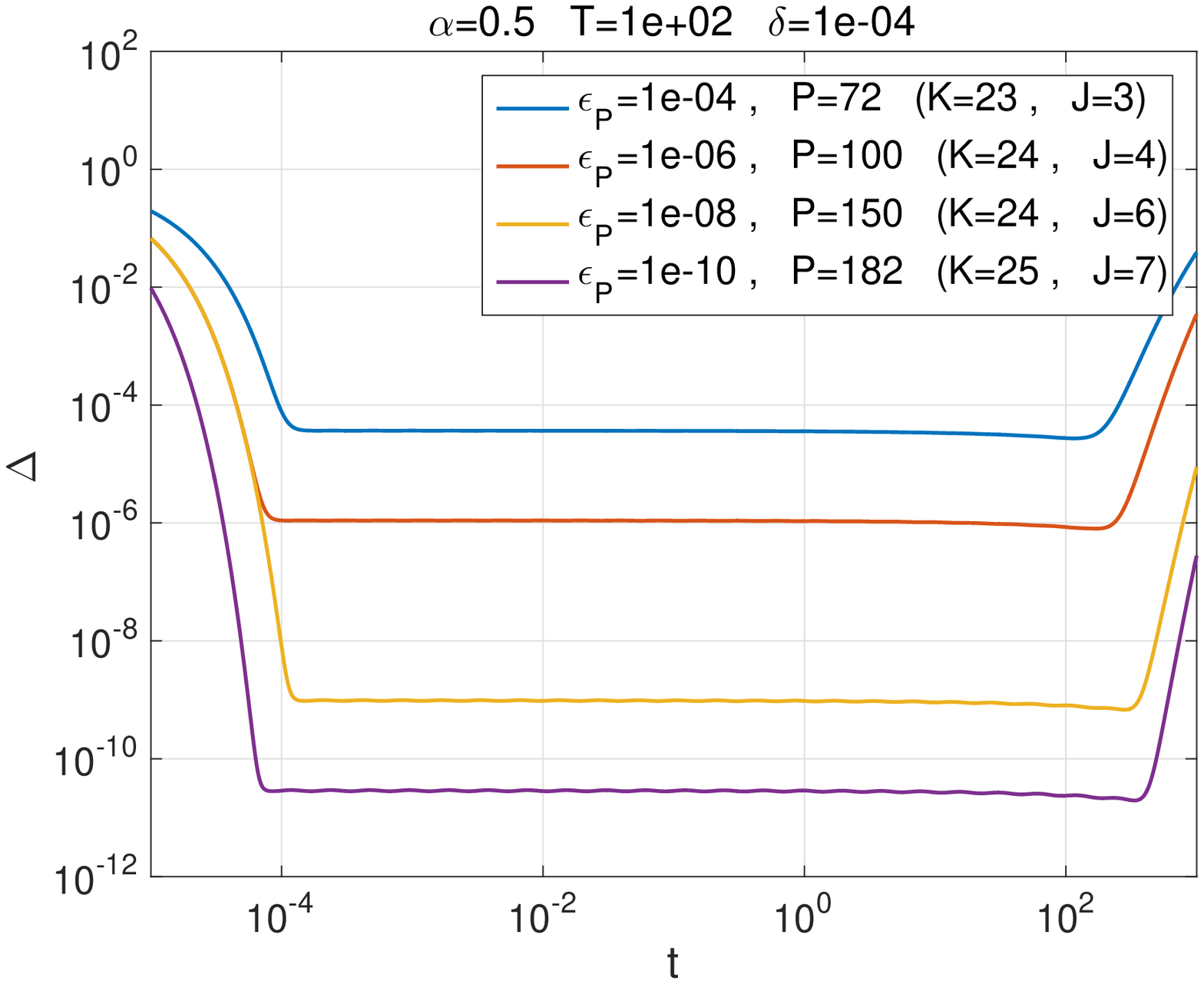}
	\caption{}
	\label{fig:KerApprox_RelErr_t_a}%
\end{subfigure}%
\hfill
\begin{subfigure}[t]{\sfw}%
	\centering
	\includegraphics[width=\ssfw,trim=0 0 30 0]{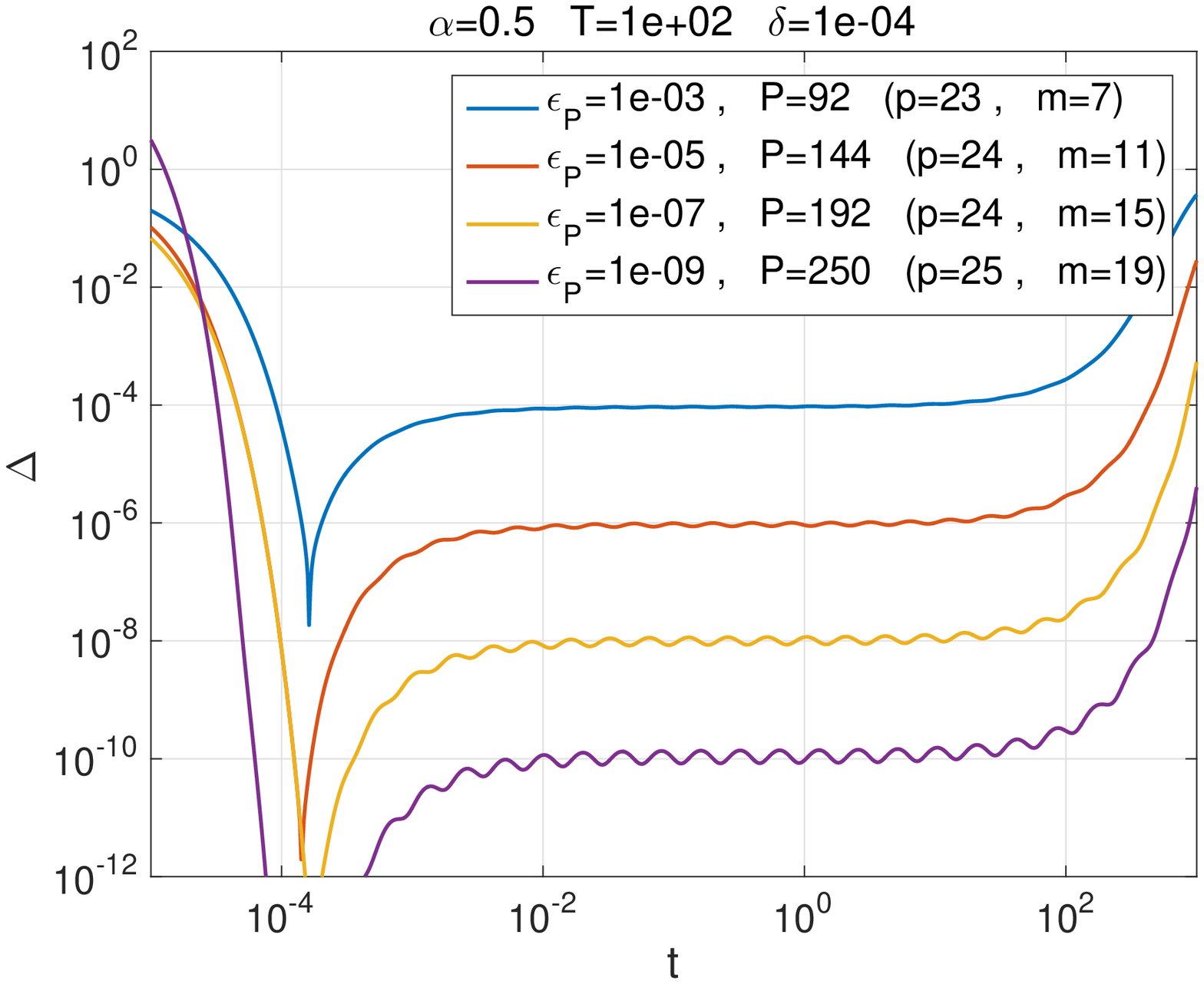}
	\caption{}
	\label{fig:KerApprox_RelErr_t_b}%
\end{subfigure}%
\caption{The relative error $\Del$, given by \deq{relErr}, in the approximation of the kernel $w$ as a function of $t$.
	\sSubref{fig:KerApprox_RelErr_t_a} Scheme \deq{eq:wdelApprox}.
	\sSubref{fig:KerApprox_RelErr_t_b} Scheme proposed in \cite{KC_FDEs}.
}
\label{fig:KerApprox_RelErr_t}
\end{figure}
Comparing the two figures, we see that in this setup, the present scheme yields slightly smaller errors with a significantly smaller number of terms.
Note that the schemes require similar numbers of intervals or circles in their respective composite quadratures, however the present scheme requires a smaller number of quadrature nodes in each interval compared to the scheme proposed in \cite{KC_FDEs} in order to achieve similar accuracy.

Figure \ref{fig:KerApprox_MRelErr_JK} shows the maximum $M$ of $\Del\frb{t}$ on a grid with a hundred points in each of the intervals $[\del 10^q\, ,\, \del 10^{q+1}]$, with $q$ integer, covering $[\del,T]$.
The results are obtained with $\a=0.01$ (top row), $\a=0.5$ (middle row), and $\a=0.99$ (bottom row).
The figures on the left show graphs of $M$ as functions $J$, with $T=10^4$, and $\del=10^{-4}$.
The different graphs in each figure are obtained with different values of $K$.
The figures on the right show graphs of $M$ as functions of $K$, with $T=10^2$, and $J=3$.
The different graphs correspond to different values of $\del$.
The dashed black lines show the prediction of the estimators $A_J$ on the left and $B_K$ on the right.
\begin{figure}
\centering
\begin{subfigure}[t]{\sfw}%
	\centering
	\includegraphics[width=\ssfw,trim=0 0 30 0]{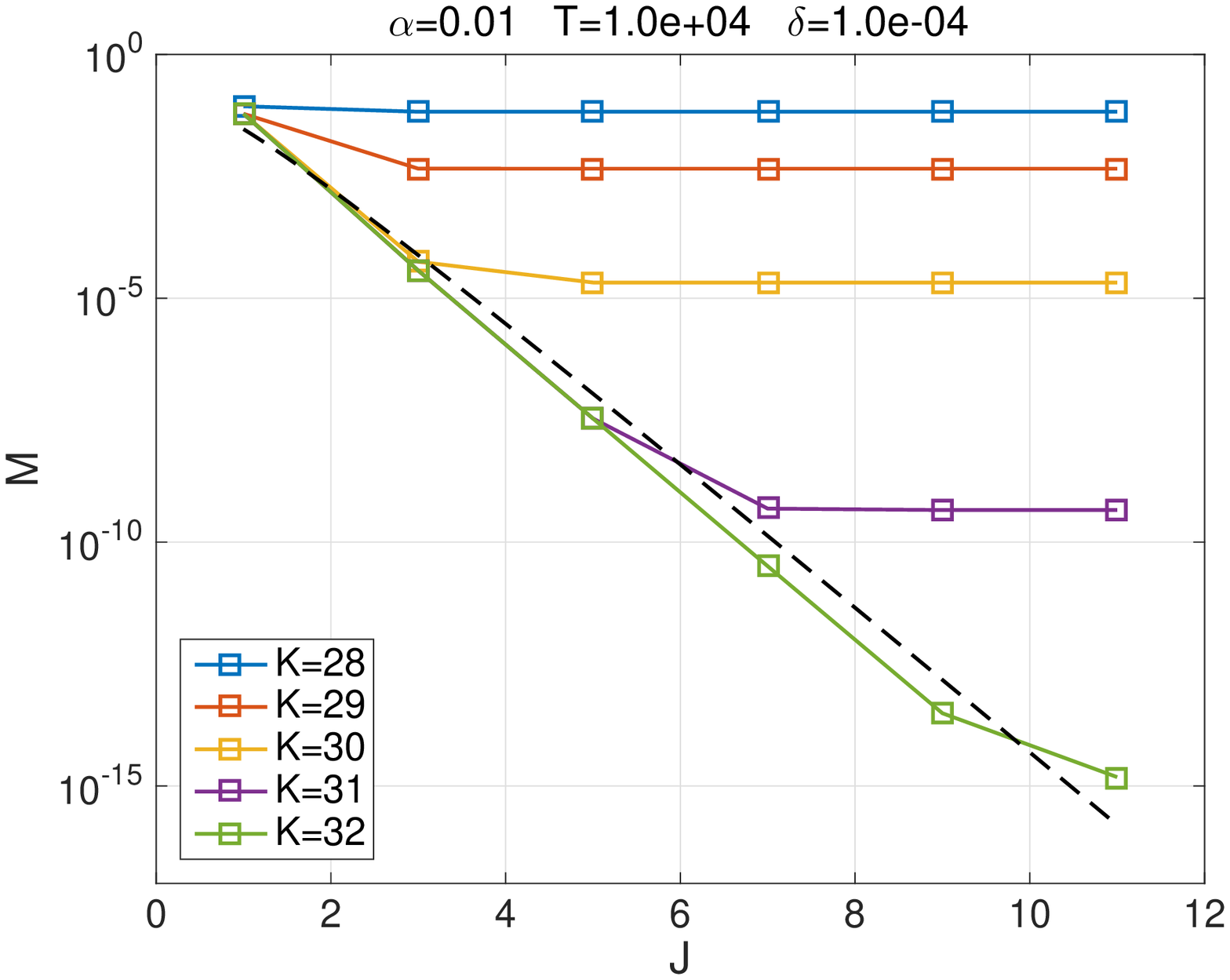}
	\caption{}
\end{subfigure}%
\hfill
\begin{subfigure}[t]{\sfw}%
	\centering
	\includegraphics[width=\ssfw,trim=0 0 30 0]{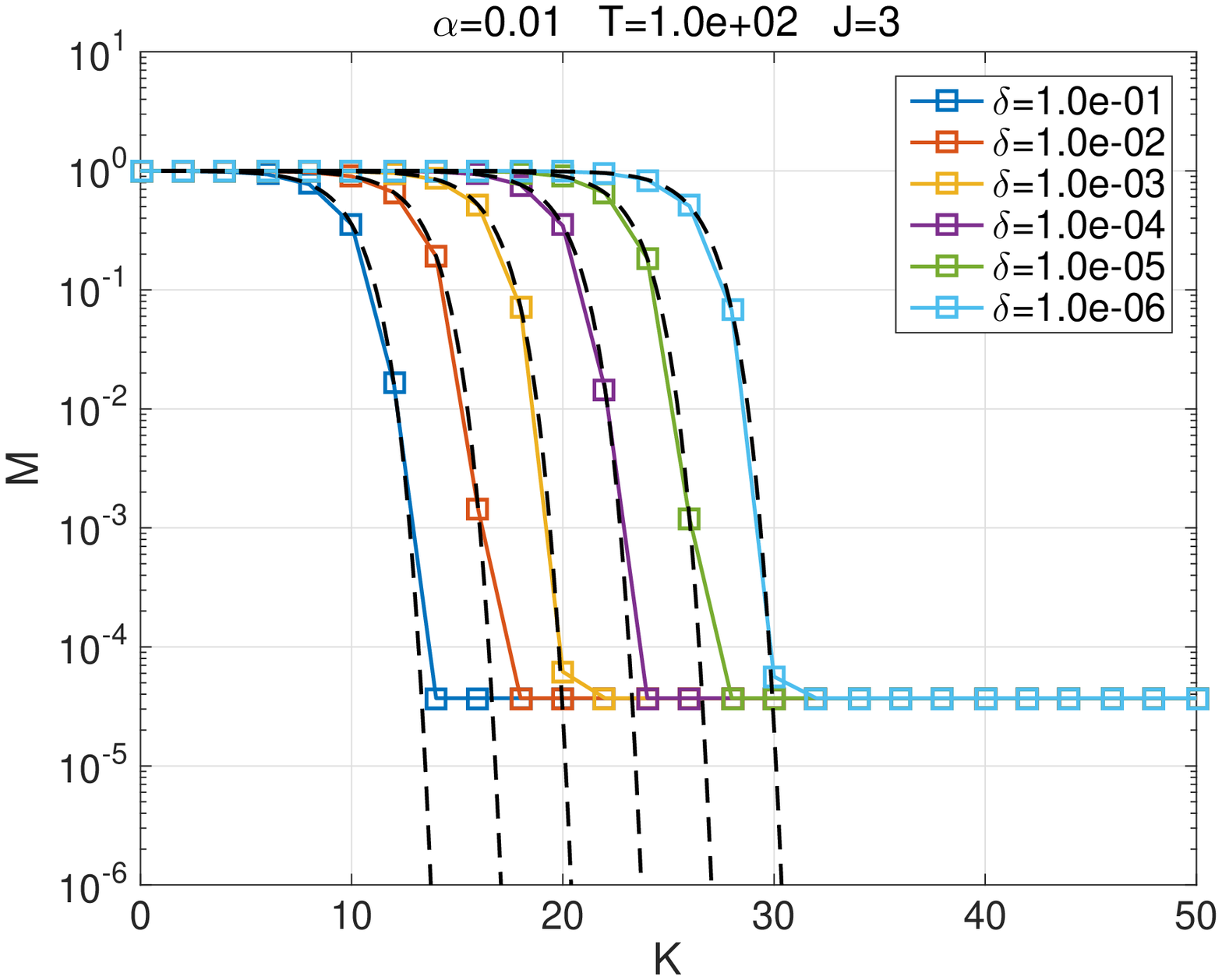}
	\caption{}
\end{subfigure}%
\\
\begin{subfigure}[t]{\sfw}%
	\centering
	\includegraphics[width=\ssfw,trim=0 0 30 0]{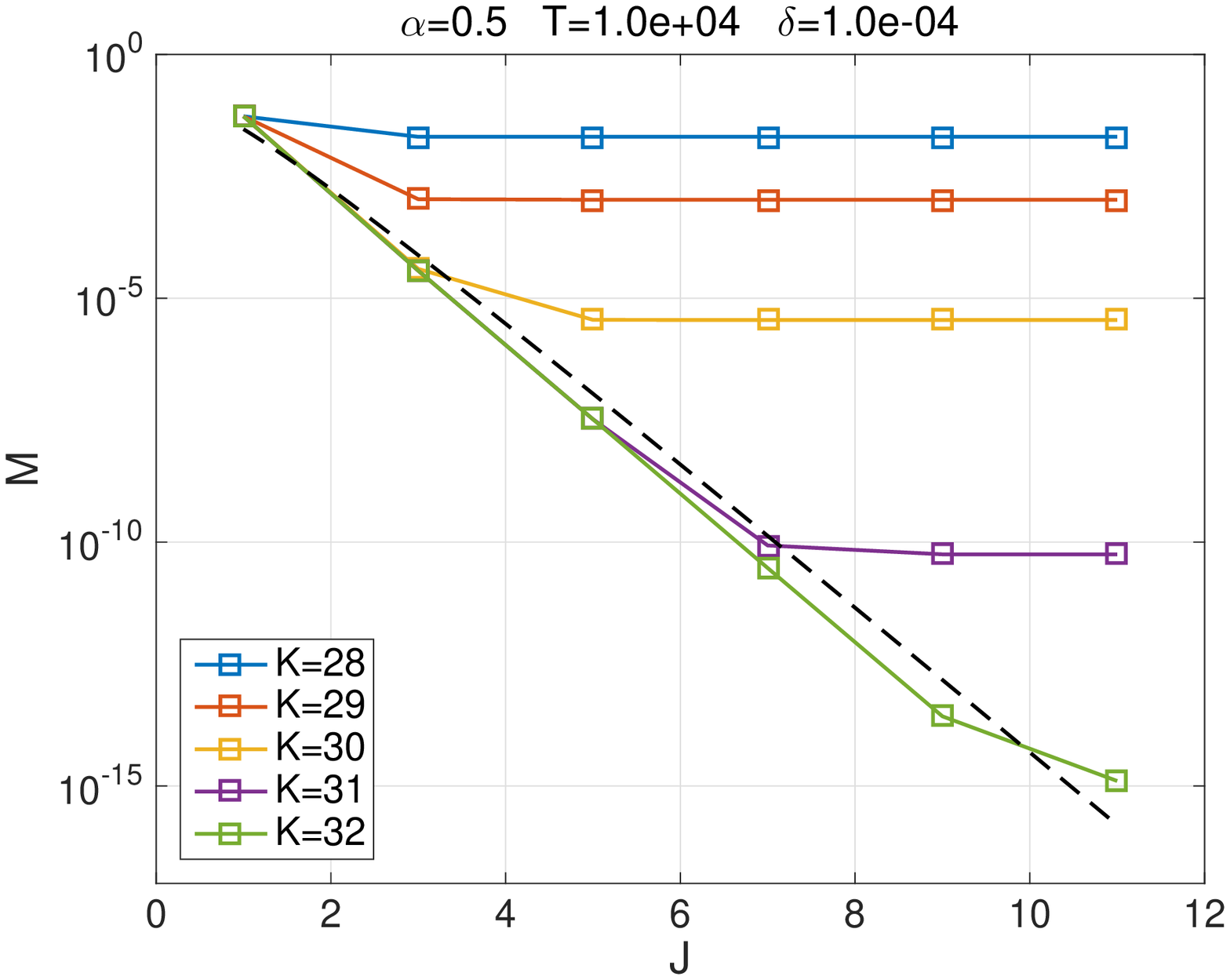}
	\caption{}
\end{subfigure}%
\hfill
\begin{subfigure}[t]{\sfw}%
	\centering
	\includegraphics[width=\ssfw,trim=0 0 30 0]{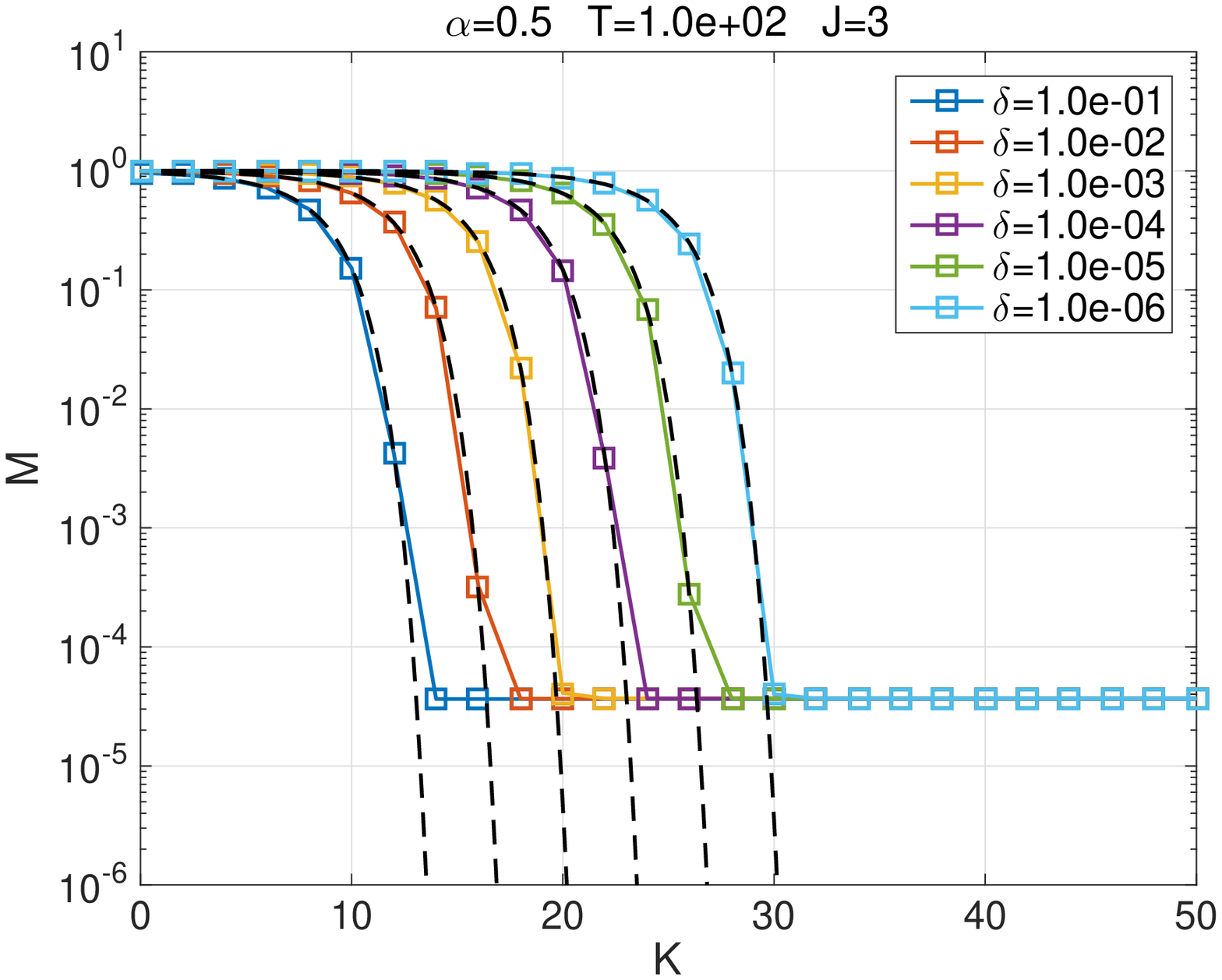}
	\caption{}
\end{subfigure}%
\\
\begin{subfigure}[t]{\sfw}%
	\centering
	\includegraphics[width=\ssfw,trim=0 0 30 0]{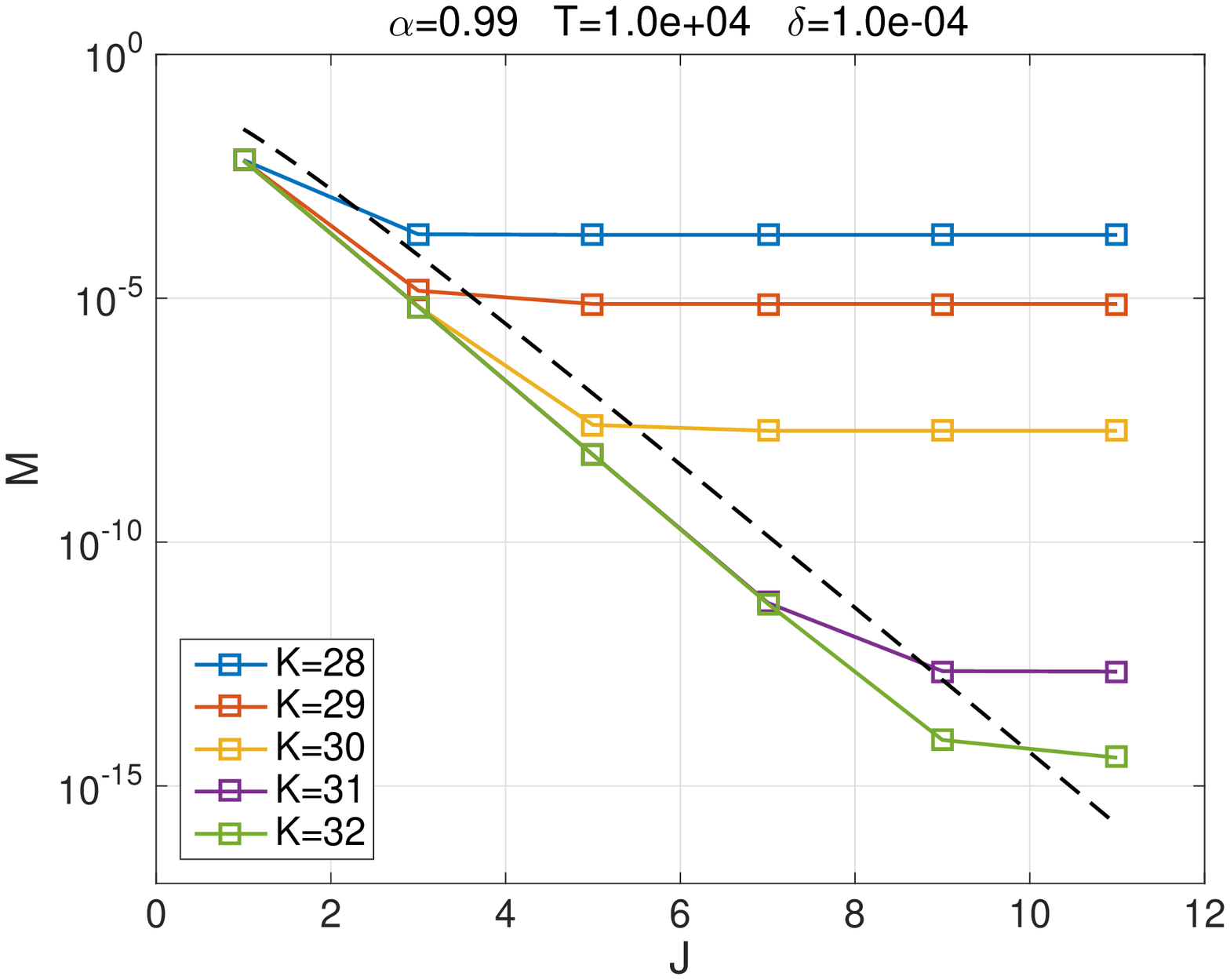}
	\caption{}
\end{subfigure}%
\hfill
\begin{subfigure}[t]{\sfw}%
	\centering
	\includegraphics[width=\ssfw,trim=0 0 30 0]{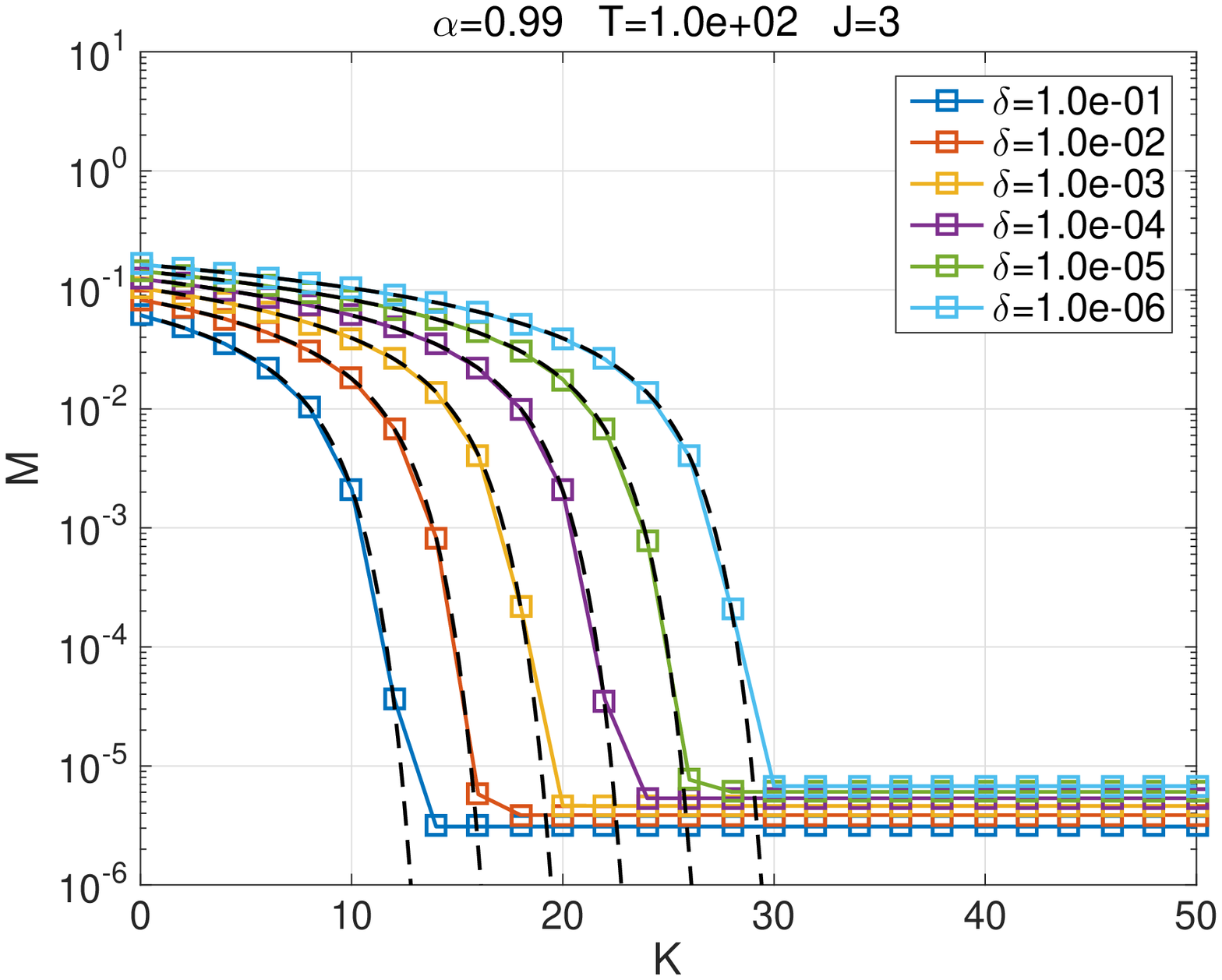}
	\caption{}
\end{subfigure}%
\caption{The maximum $M$ of the relative error $\Del$ plotted as a function of $J$ on the left and as a function of $K$ on the right.
The figures are obtained with $\a=0.01$ (top), $\a=0.5$ (middle), and $\a=0.99$ (bottom).
The dashed lines in each figure are the graphs of $A_J$ (left) and $B_K\big(\a,\del T^{-1}\big)$ (right), showing the theoretical estimate.
}
\label{fig:KerApprox_MRelErr_JK}
\end{figure}
The figures show that estimate \deq{eq:mainErrEst} predicts the error well.
The qualitative behavior of the error is also in agreement with \deq{eq:mainErrEst}.
Since both $A_J$ and $B_K$ are bounded for $\a\in(0,1)$, the number of terms required for the approximation to satisfy a prescribed error tolerance is bounded for $\a\in(0,1)$.
Comparing figures in the same column, we conclude that the results support this proposition.
Note, however that we measure smaller errors for $\a$ close to one.
Another property of \deq{eq:mainErrEst} is that the number of quadrature nodes $J$ in each interval required to satisfy an error tolerance is bounded with respect to $\a\in(0,1)$, $\del$ and $T$.
This proposition is supported by the results in the right column, which suggest that the best accuracy that can be obtained for a fixed $J$ is bounded below a certain value for all $\a$, and $\del$.

\subsection{Application to time stepping schemes}
\label{sec:AppTimeSteppScheme}

Consider the initial value problem
\beql{MLFIVP}
	D^\a u=\lam u\ ,
	\qquad
	u\frb{0}=1\ .
\eeq
Its solution is given by
\beql{ivp1Sol}
	u\frb{t}=E_\a\frb{\lam t^\a}\ ,
	\qquad
	E_\a\frb{t}=\sum_{k=0}^\infty \frac{t^k}{\Gf{\a k+1}}\ ,
\eeq
where $E_\a$ is the Mittag-Leffler function.
We approximate the solution of \deq{MLFIVP} in $(0,T)$, with $T=10$, by using \deq{eq:TR} with constant step size $\delt=10^{-3}$.
Figure \ref{fig:MLF_e} shows the results for $\lam=-1$, and $\a$ having either of the three values $\a=0.2,0.5,0.8$.
The the top left figure shows the solutions computed using \cite{MLF_Code}, and the remaining three show the error $e$ given by
\beql{eq:LocErrABS}
	e^n=|u\frb{t_n}-v^n|\ ,
\eeq
where $t_n=\delt n$, with $n=0,\ldots,N$.
In each of the latter, the different graphs correspond to different values of the parameters $K$ and $J$ controlling the approximation of the kernel.
Figure \ref{fig:MLFi_e} shows results for $\lam=i$, and $\a=0.8$.
The figure on the left shows the real and imaginary parts of the solution $u$ computed using \cite{MLF_Code}, and the figure on the right shows the error $e$.
In both figures \ref{fig:MLF_e} and \ref{fig:MLFi_e}, the graphs corresponding the highest number $P$ of auxiliary variables, seem to be saturated by the discretization error.
That is, at that point, improving the accuracy of the approximation of the kernel does not reduce the error anymore.
This demonstrates the efficiency of the scheme which requires a small number of auxiliary variables to account for the history term.
While in all the tests $P\le 100$ is sufficient to account for the history term, the time stepping scheme requires $N=10^4$ steps to compute the approximation in the interval of interest $(0,T)$.
\begin{figure}
\centering
\begin{subfigure}[t]{\sfw}%
	\centering
	\includegraphics[width=\ssfw,trim=0 0 30 0]{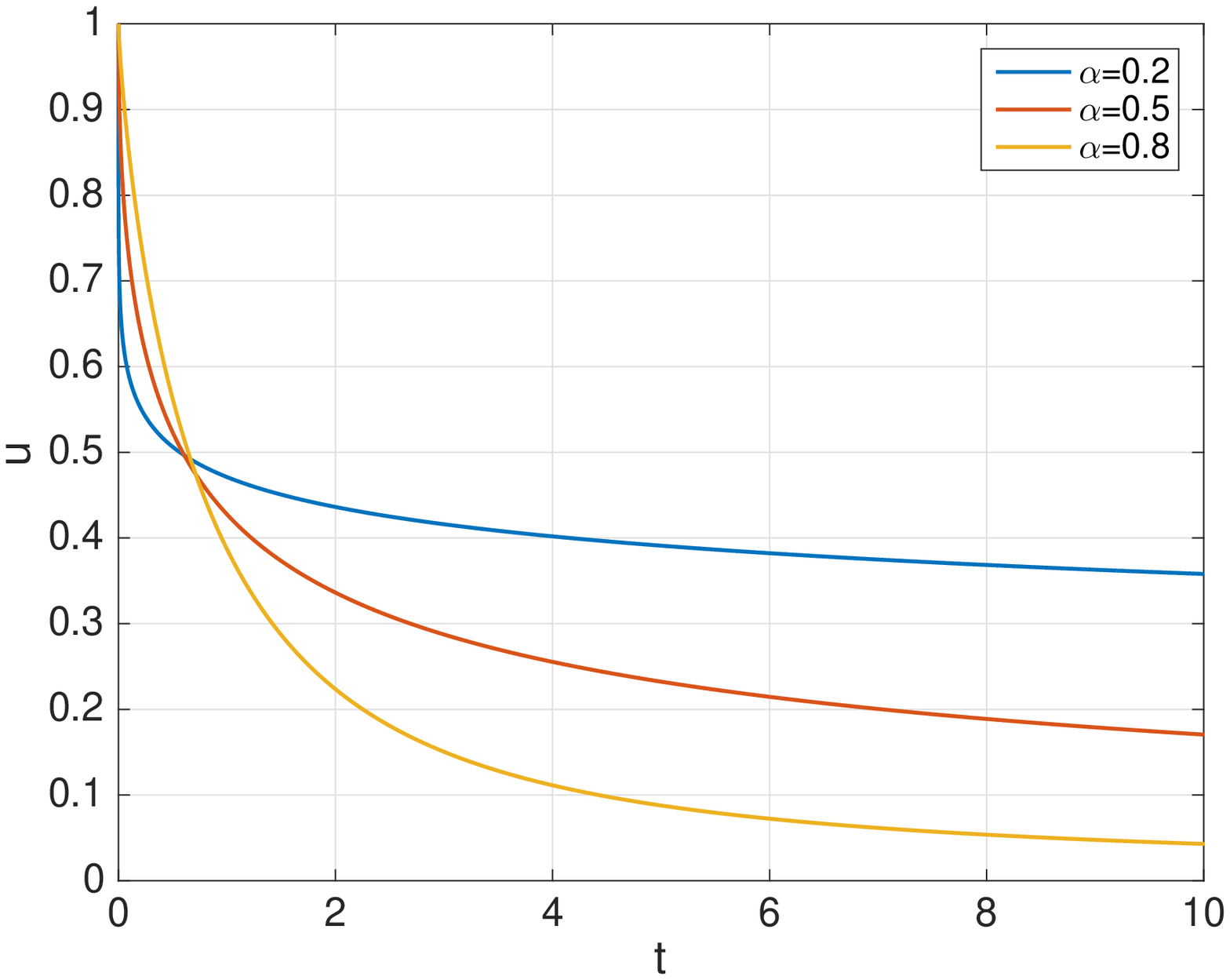}
	\caption{}
	\label{fig:MLF-1Sol}
\end{subfigure}%
\hfill
\begin{subfigure}[t]{\sfw}%
	\centering
	\includegraphics[width=\ssfw,trim=0 0 30 0]{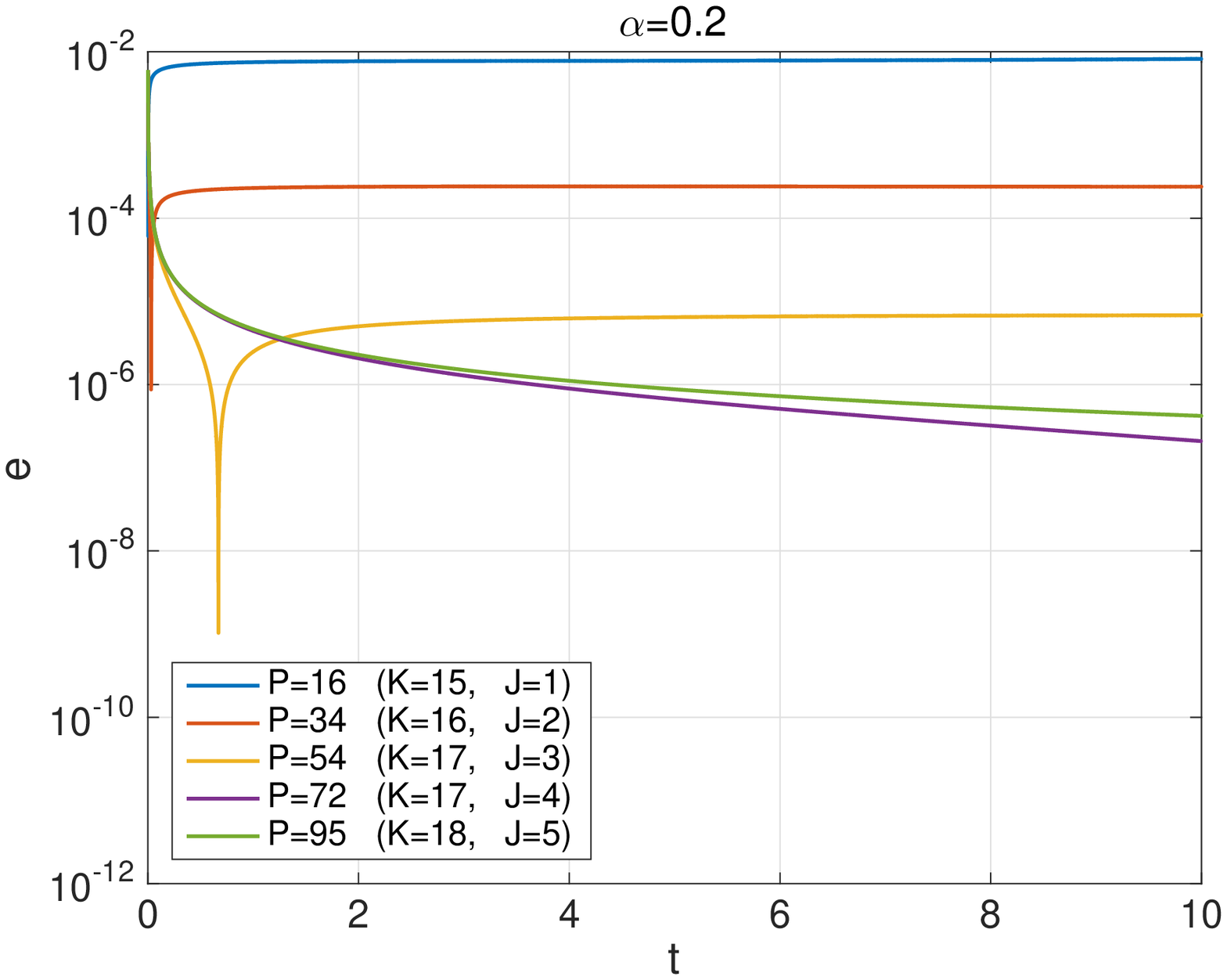}
	\caption{}
	\label{fig:MLF-1_1}
\end{subfigure}%
\\
\begin{subfigure}[t]{\sfw}%
	\centering
	\includegraphics[width=\ssfw,trim=0 0 30 0]{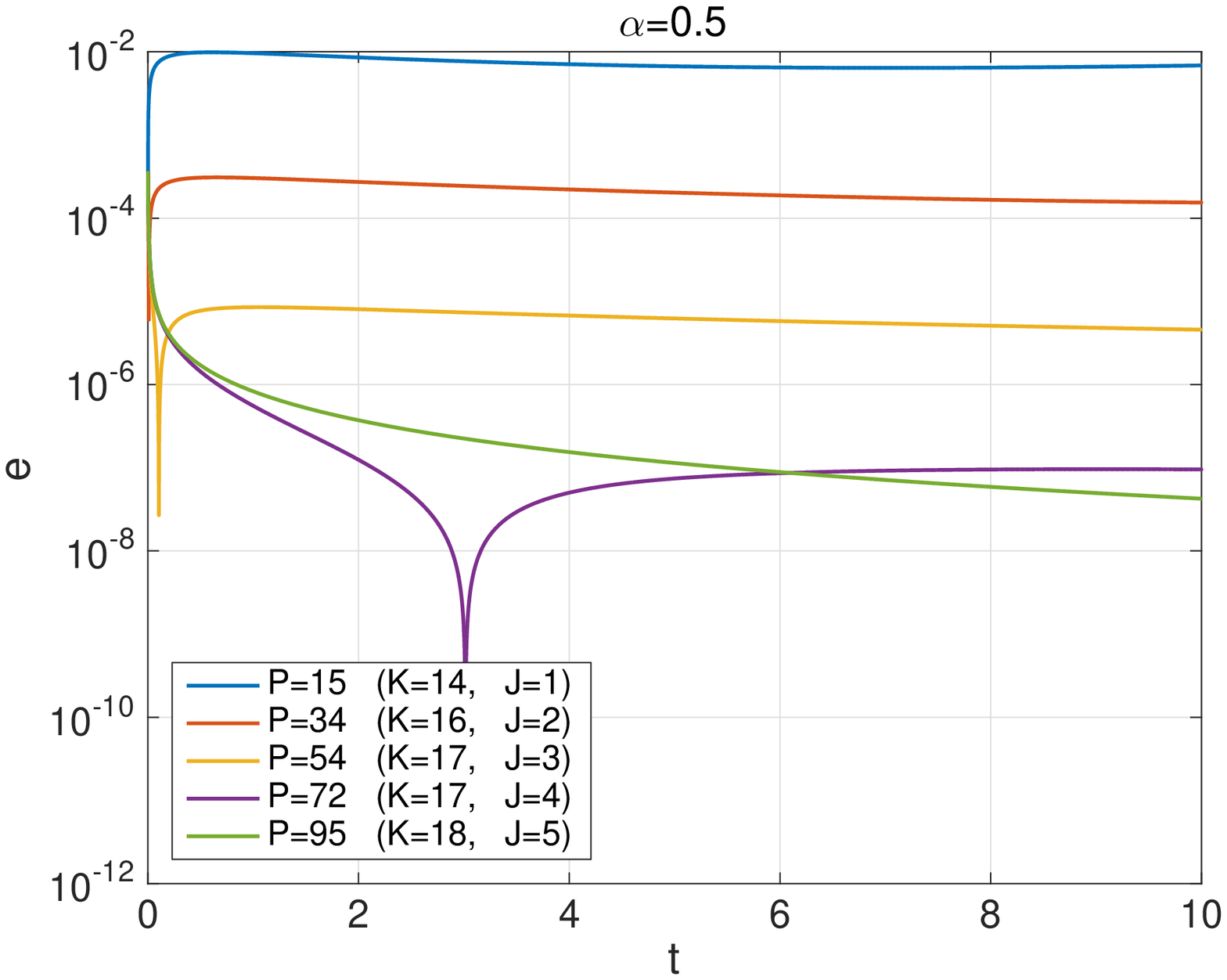}
	\caption{}
	\label{fig:MLF-1_2}
\end{subfigure}%
\hfill
\begin{subfigure}[t]{\sfw}%
	\centering
	\includegraphics[width=\ssfw,trim=0 0 30 0]{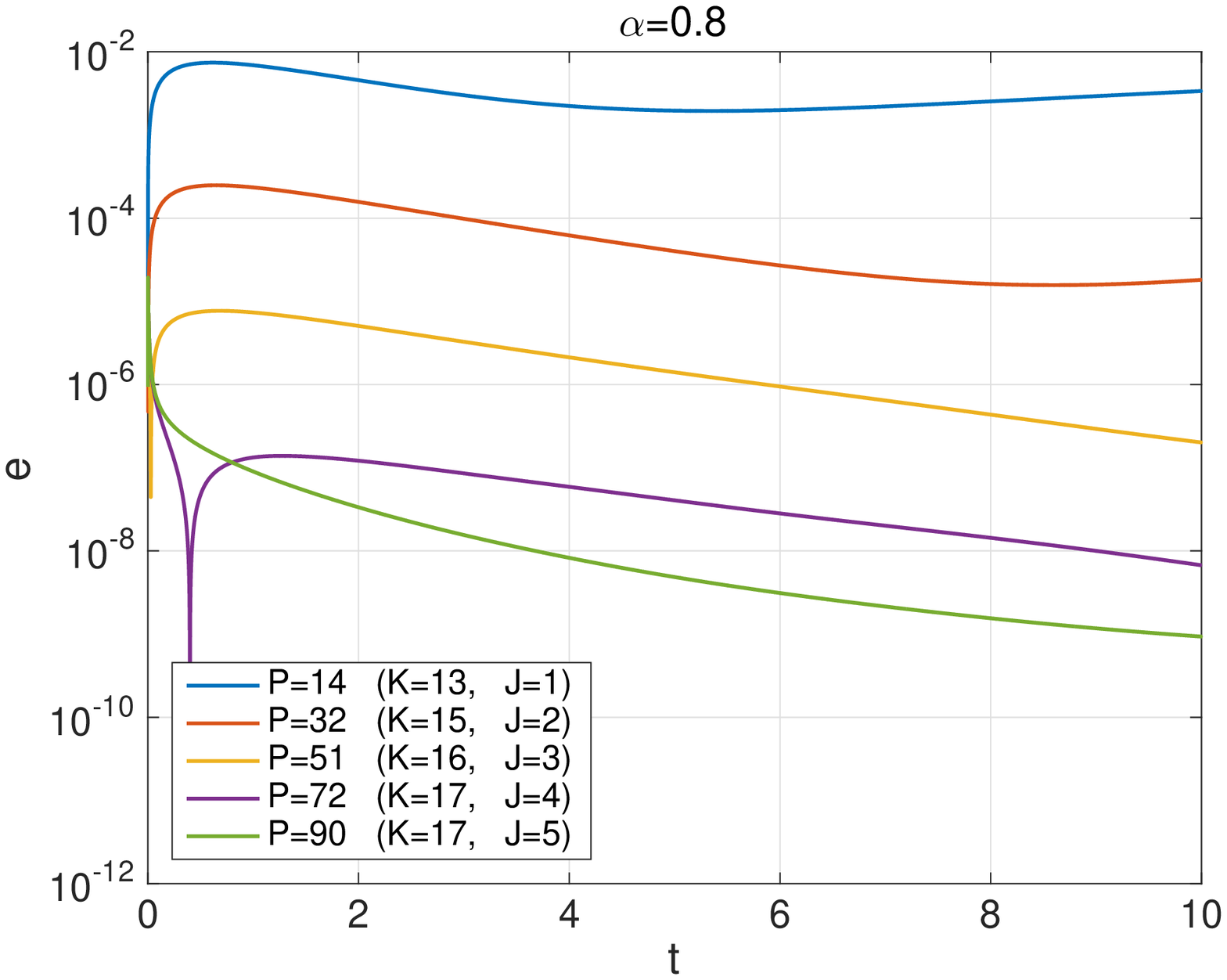}
	\caption{}
	\label{fig:MLF-1_3}
\end{subfigure}
\caption{Problem \deq{MLFIVP} with $\lam=-1$.
	\sSubref{fig:MLF-1Sol} Solutions for $\a=0.2$, $0.5$, and $0.8$.
	\sSubref{fig:MLF-1_1}--~\sSubref{fig:MLF-1_3} The error $e$ as a function of $t$,
	for $\a=0.2$, $0.5$, and $0.8$, respectively.}
\label{fig:MLF_e}
\end{figure}
\begin{figure}
\centering
\begin{subfigure}[t]{\sfw}%
	\centering
	\includegraphics[width=\ssfw,trim=0 0 30 0]{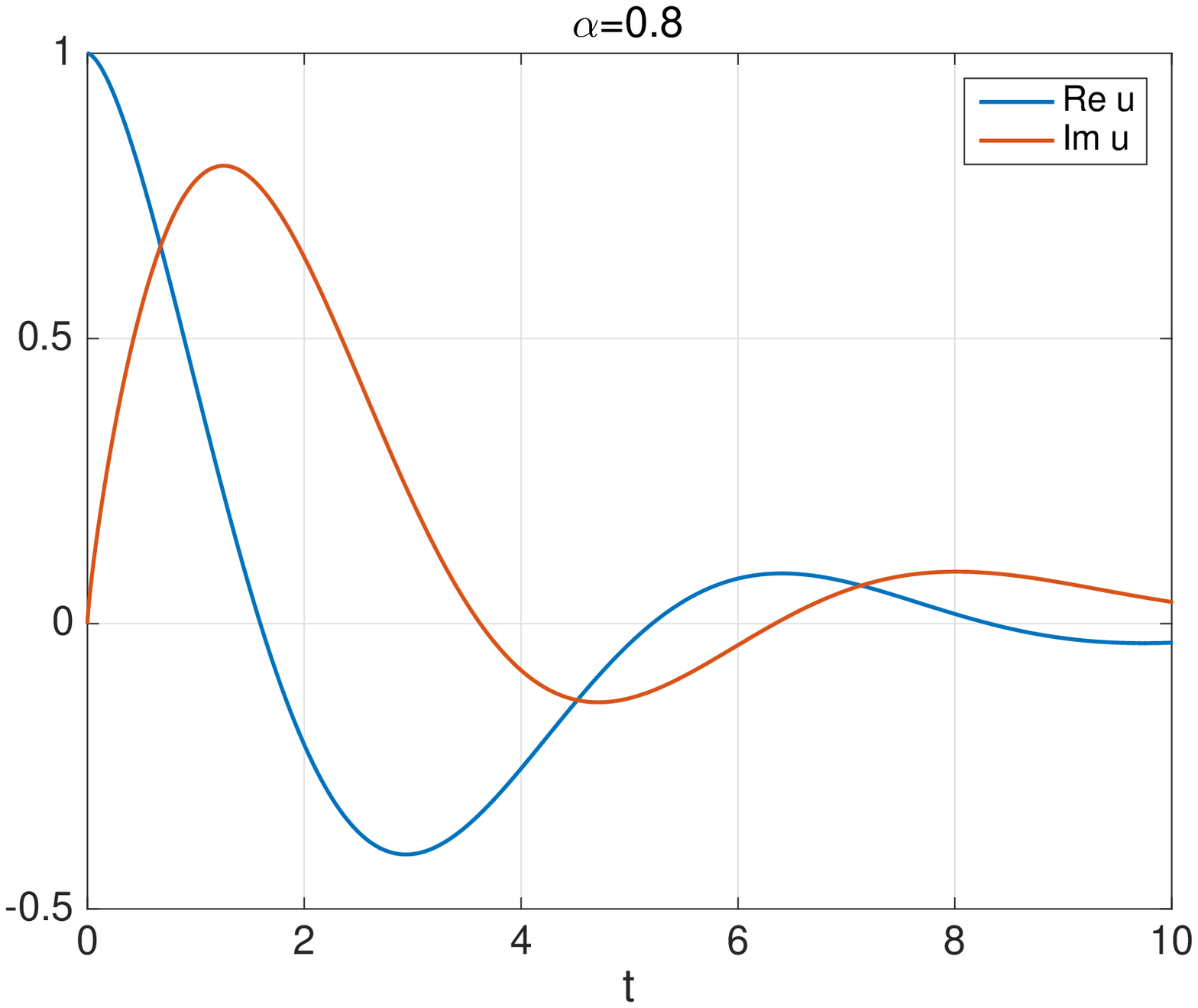}
	\caption{}
	\label{fig:MLFi}
\end{subfigure}%
\hfill
\begin{subfigure}[t]{\sfw}%
	\centering
	\includegraphics[width=\ssfw,trim=0 0 30 0]{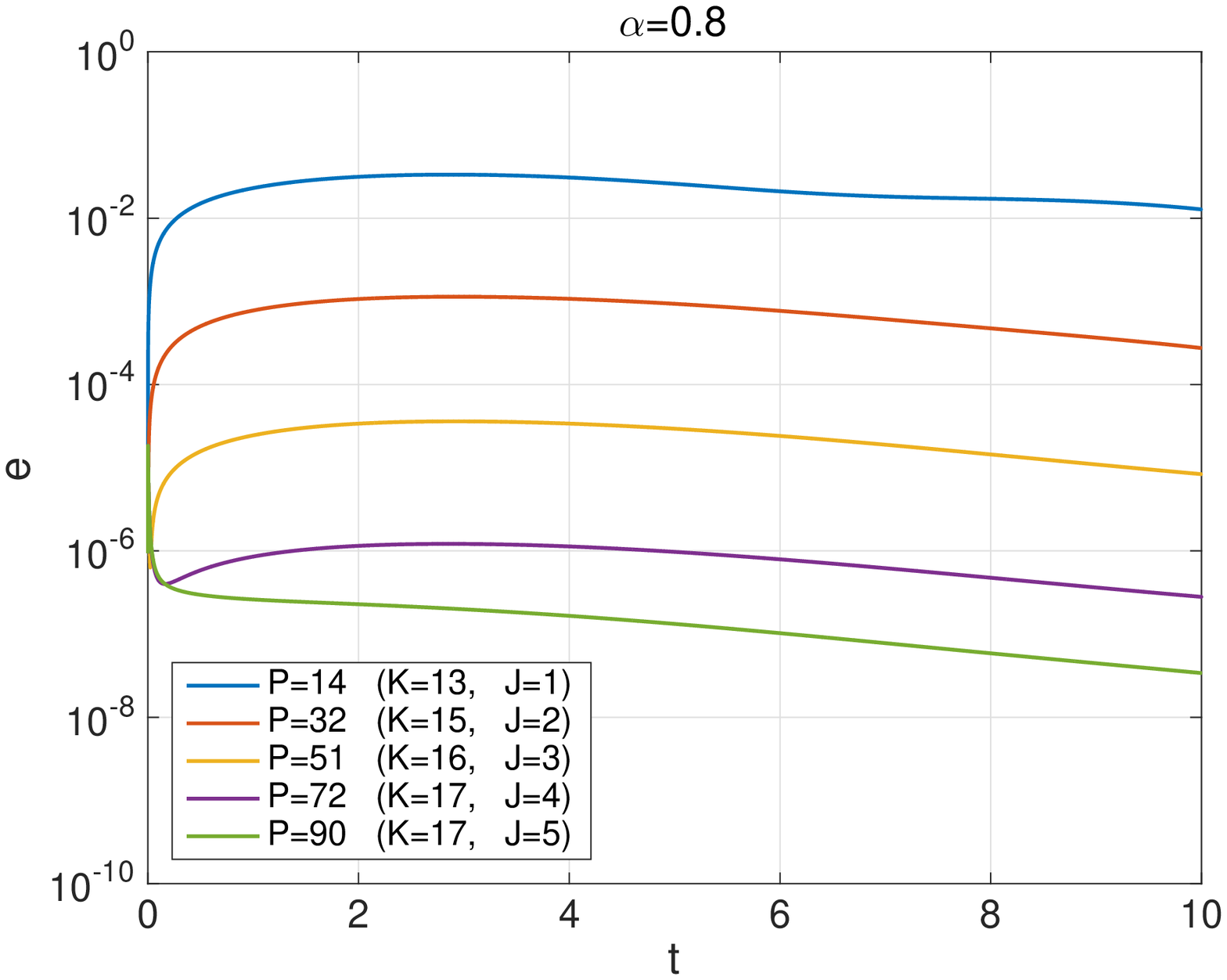}
	\caption{}
	\label{fig:MLFi_1}
\end{subfigure}%
\caption{Problem \deq{MLFIVP} with $\lam=i$ and $\a=0.8$.
	\sSubref{fig:MLFi} Real and imaginary parts of the solution $u$.
	\sSubref{fig:MLFi_1} The error $e$ as a function of $t$.}
\label{fig:MLFi_e}
\end{figure}

Next we test the performance of the kernel compression scheme as a part of a high order adaptive time stepping method, denoted LER-IDC (for details, see Section \ref{sec:SteppScheme}).
Consider the nonlinear fractional differential equation
\bseql{vdp2IVP}
\beql{vdp2}
	\rb{D^\a}^2 x -\mu\rb{1-x^2}D^\a x +x=0
\eeq
in $\rb{0,T}$, with initial conditions
\beq
	x\frb{0}=x_0\ , \qquad D^\a x\frb{0}=y_0\ .
\eeq
\eseq
Here $\mu$ is a non-negative constant, and $x_0$, $y_0\in\R$.
For $\a=1$, \deq{vdp2} is reduced to the classical Van der Pol equation which can be shown to have a stable periodic solution.
To apply the scheme we write \deq{vdp2} as a system by substituting $y=D^\a x$.
Thus, we have
\bseql{vdpIVP}
\begin{align}
\label{vdpx}
	D^\a x &=y \\
\label{vdpy}
	D^\a y &=\veps\rb{1-x^2}y -x\ ,
\end{align}
in $\rb{0,T}$ subject to the initial condition
\beql{incondsys}
	x\frb{0}=x_0\ , \qquad y\frb{0}=y_0\ .
\eeq
\eseq
In the tests below, $\a=0.8$, $\mu=4$, $x_0=2$, $y_0=0$, and $T=25$.
To measure the accuracy of approximations we compare them to a reference numerical solution.
The reference solution $v_\tref=(x_\tref,y_\tref)^T$, shown in Figure \ref{fig:FracVDP1}, is computed with the an adaptive TR-IDC scheme \cite{KC_Adapt}, with error tolerance $\veps=10^{-10}$ and the kernel compression scheme of \cite{KC_FDEs}.

Figure \ref{fig:FracVDP2} shows, starting at the top right and continuing counterclockwise, the number $P$ of auxiliary variables, the step size $\delt$, and the local error
\beq
	e_h^n=|v^n-v_\tref\frb{t_n}|_1\ ,
\eeq
where $v_\tref=v_\tref\frb{t}$ is the spline interpolation of reference solution $v_\tref$ at $t$.
The different graphs in each figure are obtained with different error tolerances $\veps_\delt$.
Comparing Figure \ref{fig:FracVDP2} and Figure \ref{fig:FracVDP1}, we see that the program is able to detect changes in the behavior of the solution, and change the step size and the number of auxiliary variables accordingly.
Figure \ref{fig:FracVDP3} shows the global error
\beq
	E_1=\sum_{n=1}^N e_h^n\delt_{n-1}
\eeq
as a function of the average step size
\beq
	\delt_{\mathrm{avg}}=\frac{1}{N}\sum_{n=1}^N \delt_{n-1}=\frac{T}{N}\ ,
\eeq
where for each $n$, $\delt_n=t_{n+1}-t_{n}$.
Thus we measure 4th-order convergence of the global error as a function of the average step size, similarly to the results presented in \cite{KC_Adapt}.
\begin{figure}
\centering
\includegraphics[width=\sfw,trim=0 0 30 0]{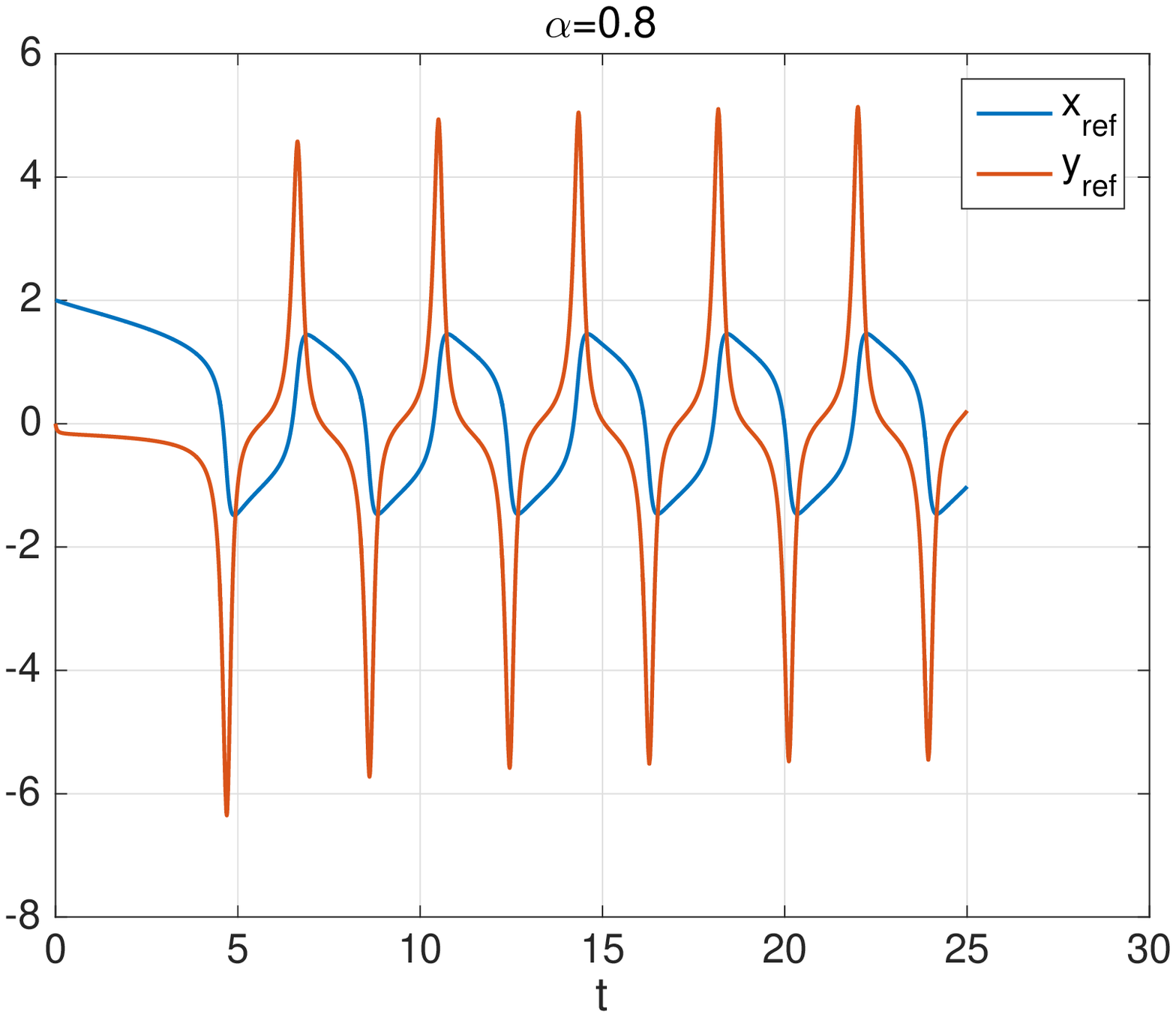}
\caption{Fractional Van der Pol equation: reference solution.}
\label{fig:FracVDP1}
\end{figure}
\begin{figure}
\centering
\begin{subfigure}[t]{\sfw}%
	\centering
	\includegraphics[width=\ssfw,trim=0 0 30 0]{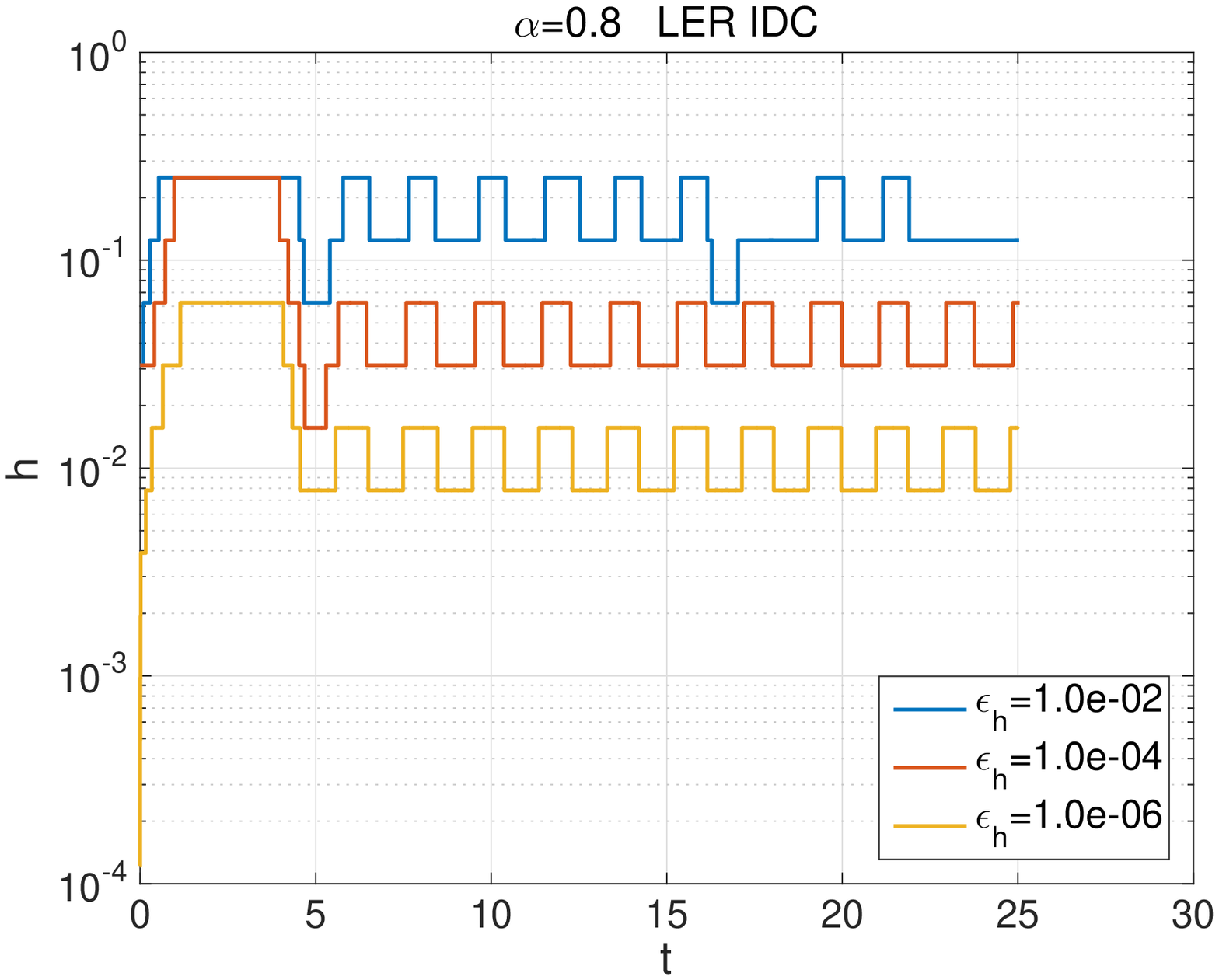}
	\caption{}
	\label{fig:VDP2_1}
\end{subfigure}%
\hfill
\begin{subfigure}[t]{\sfw}%
	\centering
	\includegraphics[width=\ssfw,trim=0 0 30 0]{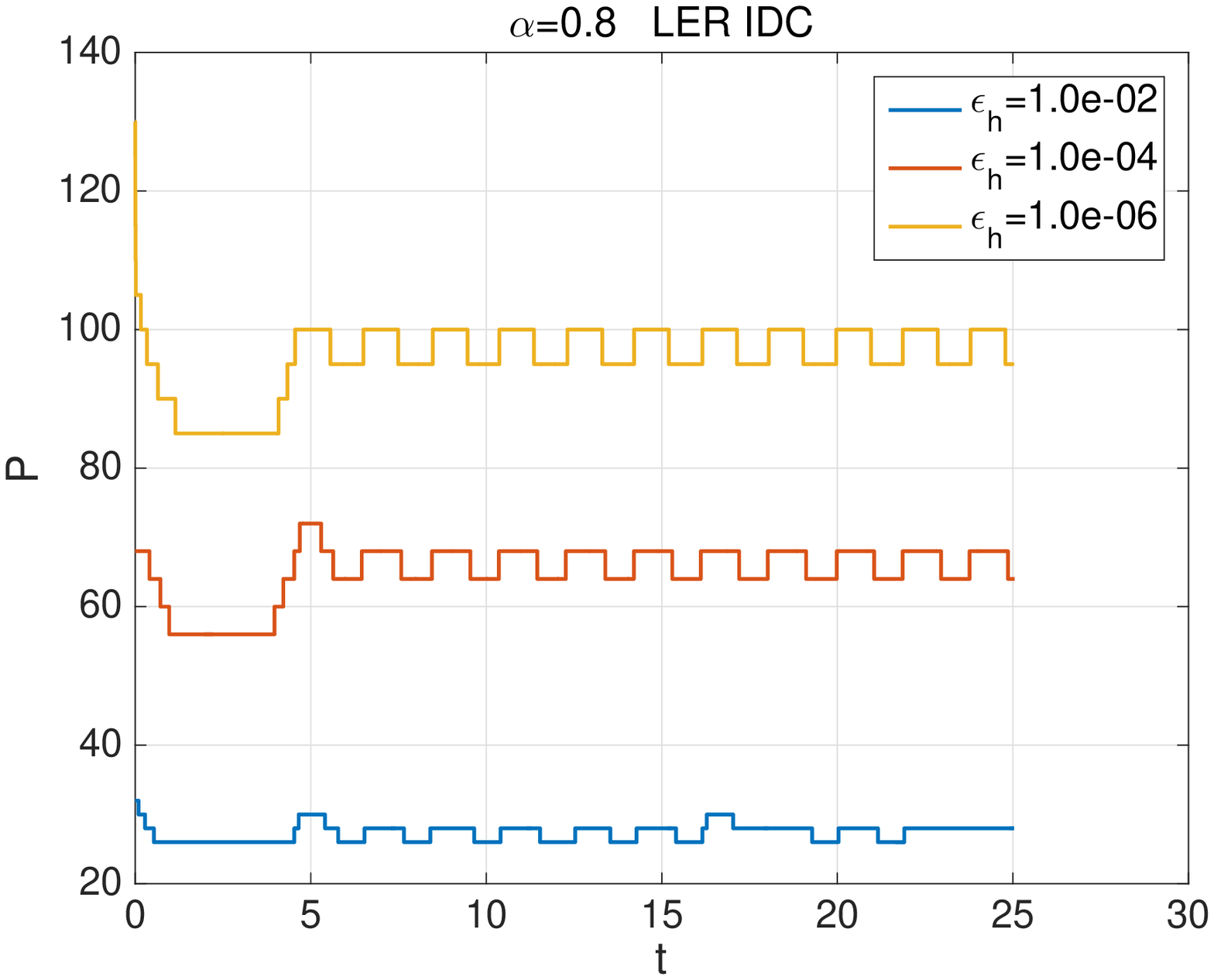}
	\caption{}
	\label{fig:VDP2_2}
\end{subfigure}%
\\
\begin{subfigure}[t]{\sfw}%
	\centering
	\includegraphics[width=\ssfw,trim=0 0 30 0]{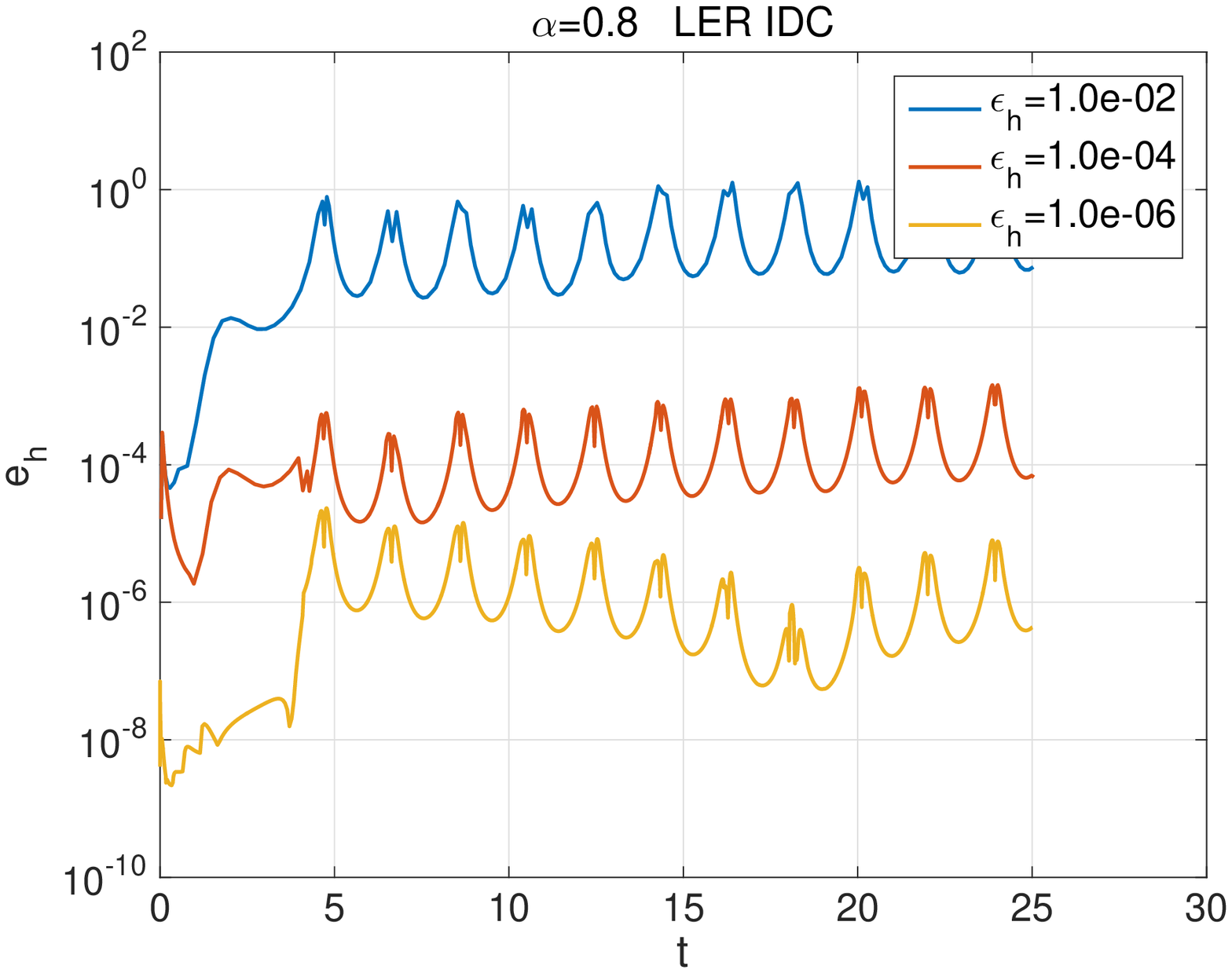}
	\caption{}
	\label{fig:VDP2_3}
\end{subfigure}%
\hfill{}
\caption{Fractional Van der Pol equation.
	\sSubref{fig:VDP2_1} The step size $\delt$ as a function of $t$.
	\sSubref{fig:VDP2_2} The number $P$ of auxiliary variables a function of $t$.
	\sSubref{fig:VDP2_3} The local error $e_h$ as a function of $t$.}
\label{fig:FracVDP2}
\end{figure}
\begin{figure}
\centering
\includegraphics[width=\sfw,trim=0 0 30 0]{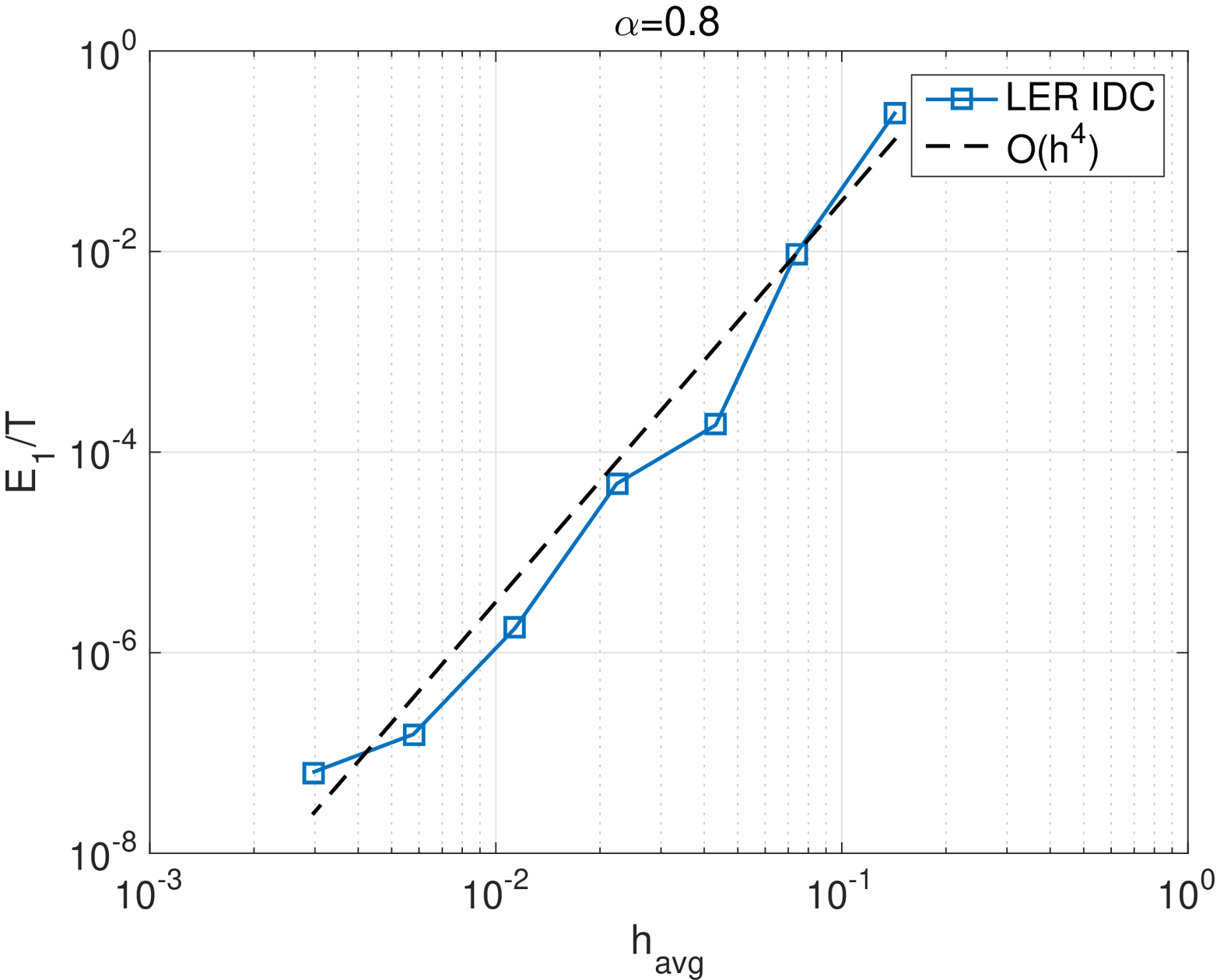}
\caption{Fractional Van der Pol equation; the global error $E_1$ as a function of the average step size $\delt_\mathrm{avg}$.}
\label{fig:FracVDP3}
\end{figure}

\section{Concluding remarks}
\label{sec:Conc}

A scheme is proposed for approximating the kernel $w_\del$ of the the history term of the fractional integral by a linear combination of exponentials \deq{eq:wdelApprox}.
Numerical results are presented, testing the approximation of the kernel directly and the application of the scheme as a part of two fully discrete time stepping methods: a low order method, and a high order fully adaptive method.

The kernel compression scheme is based on the application of a composite Gauss-Jacobi quadrature rule to an integral representation of the kernel of the fractional integral $w$.
We obtain estimate \deq{eq:mainErrEst} which shows that the error converges rapidly, as expected, in the number $J$ of quadrature nodes associated with each interval of the composite rule.
According to \deq{eq:mainErrEst}, the number $P=(K+1)J$ of terms required to satisfy a prescribed error tolerance is bounded for $\a\in(0,1)$.
Moreover, $J$ depends only on the error tolerance.
The numerical results support the validity of \deq{eq:mainErrEst} and show similar qualitative behavior.

The discussion above, however, should be considered with some care.
Since we currently do not have convergence analyses for time stepping methods that employ kernel compression schemes, it is not clear what type of estimates kernel compression schemes must satisfy in order to obtain convergence (in any sense).
This should be taken into account, since in the statements above, we rely on estimate \deq{eq:mainErrEst}, where the term $A_J$ associated with the number of quadrature nodes in each interval is independent of the parameters of the problem, while the scheme proposed in \cite{KC_FDEs}, for example, satisfies an estimate of the type \deq{Intro:RelL2Est} for any $f\in L^2\frb{0,T}$, with a similar property.

The above mentioned properties of the scheme, combined with the observation that the nodes $a_{kj}$, with $k\ge0$, and $J=1,\ldots,J$ are independent of $\del$, suggest that, similarly to the scheme proposed in \cite{KC_FDEs}, the present scheme has a modular structure which makes it convenient for use within adaptive step-size schemes \cite{KC_Adapt}.
Numerical tests performed with a high order adaptive step-size method incorporating the proposed kernel compression scheme show promising results.
The scheme is able to detect changes in the behavior of the solution and adapt the step size and the approximation of the kernel accordingly.

\appendix

\section{A technical lemma}

Estimate \deq{eq:W1_est} of Lemma \ref{lem:W1approx} relies on the infimum of
\beq
	R_J\frb{\ell}=\rb{3-\ell}^{-1}\rb{\ell+\sqrt{\ell^2-1}}^{-2J}\ ,
\eeq
for $\ell\in(1,3)$.
It is a simple exercise to show that for each $J\ge1$, $R_J$ has a unique minimum point in $(1,3)$ and to estimate the asymptotic behavior of that minimum, at the limit $J\to\infty$.
This is summarized in the following lemma.

\begin{lemma}\label{lem:estRJlJ}
For each $J\ge1$, $R_J$ has a unique minimizer $\ell_J$ in $(1,3)$ given by
\beq
	\ell_J=\frac{3-\mu\sqrt{8+\mu^2}}{1-\mu^2}\ ,
	\qquad
	\mu=\frac{1}{2J}\ .
\eeq
In particular $\ell_J\in(3/2,3)$, and at the limit $J\to\infty$, there holds
\beq
	R_J\frb{\ell_J} \sim \frac{\me}{\sqrt{2}}\, J\rb{3+\sqrt{8}}^{-2J}\ .
\eeq
\end{lemma}

\begin{proof}
The derivative of $R_J$ is given by
\beq
	R_J'\frb{\ell}=\rb{3-\ell}^{-2}\frac{\rb{\ell+\sqrt{\ell^2-1}}^{-2J}}{\sqrt{\ell^2-1}}
		\rb{\sqrt{\ell^2-1}-2J\rb{3-\ell}}\ .
\eeq
Thus, the critical points of $R_J$, satisfy
\beq
	\mu\sqrt{\ell^2-1}=3-\ell\ ,
	\qquad
	\mu=\frac{1}{2J}\ .
\eeq
It is therefore clear that $R_J$ has a unique critical point and that this point is the unique minimizer of $R_J$ in $(1,3)$.
By squaring the last equality, we find that $\ell=\ell_J$ is a solution of the following equation
\beq
	\rb{1-\mu^2}\ell^2-6\ell+9+\mu^2=0\ .
\eeq
This equation has only one solution smaller than three given by
\beq
	\ell_J=\frac{3-\mu\sqrt{8+\mu^2}}{1-\mu^2}\ ,
	\qquad
	\mu=\frac{1}{2J}\ .
\eeq
Indeed, there holds
\beq
	\frac{3}{2} < 3-\frac{\sqrt{8+1/4}}{2}< \frac{3-\mu\sqrt{8+\mu^2}}{1-\mu^2}
		<\frac{3-\mu\sqrt{8\mu^2+\mu^2}}{1-\mu^2}=3\ .
\eeq
At the limit where $J$ tends to infinity, we have
\beq
	\ell_J=3-\mu\sqrt{8}+3\mu^2+O\frb{\mu^3}\ ,
	\qquad
	\mu=\frac{1}{2J}
\eeq
and therefore
\beq
	R_J\frb{\ell_J} = \frac{\mu^{-1}}{\sqrt{8}\brb{1+O\frb{\mu}}}
			\rb{\rho-\rho\mu +O\frb{\mu^2}}^{-1/\mu}
\eeq
where $\rho=3+\sqrt{8}$.
Thus we recover
\beqa
	R_J\frb{\ell_J} & \sim \frac{J}{\sqrt{2}}\rho^{-2J}
			\rb{1-\mu}^{-1/\mu}
		\sim \frac{\me}{\sqrt{2}}\, J\rho^{-2J}
\eeqa
which completes the proof.
\end{proof}


\end{document}